\documentclass[12pt]{amsart}
\usepackage{amssymb,amsmath,amsfonts,amsthm,bbm,mathrsfs,times,graphicx,color,comment,mathabx,enumerate}
\usepackage[utf8]{inputenc}
\usepackage{tikz}
\usetikzlibrary{arrows,shapes}
\usetikzlibrary{fadings}
\usetikzlibrary{intersections}
\usetikzlibrary{decorations.pathreplacing}
\usetikzlibrary{quotes}
\usetikzlibrary{angles}
\usetikzlibrary{babel}

\usepackage{mathtools,mathrsfs,times,subfigure,graphicx}
\usepackage{epsfig,psfrag,epsf}
\usepackage{xcolor, array}

\usepackage[super]{nth}

\usetikzlibrary{shapes.misc}
 \usetikzlibrary{decorations.markings}
\tikzset{cross/.style={cross out, draw=black, minimum size=2*(#1-\pgflinewidth), inner sep=0pt, outer sep=0pt},
cross/.default={1pt}}

\usepackage{calc}
\usepackage{ifthen}

\usepackage{extarrows}

\newcommand\numberthis{\addtocounter{equation}{1}\tag{\theequation}}

\usepackage{mathtools}

\DeclareMathOperator{\re}{\mathbb{R}e}
\DeclareMathOperator{\im}{\mathbb{I}m}
\DeclareMathOperator{\sgn}{\mbox{sgn}}

\newcommand{\INF}{{\infty}}

\newcommand{\supp}{\mbox{supp}}

\newcommand{\tta}{\theta}

\newcommand{\OM}{\Omega}

\newcommand{\sph}{{{\mathbf S}^ 1}}

\newcommand{\del}{\partial}

\newcommand{\Gam}{\varGamma}
\newcommand{\ol}{\overline}
\newcommand{\ds}{\displaystyle}
\newcommand{\dba}{\overline{\partial}}

\newcommand{\BR}{\mathbb{R}}
\newcommand{\BC}{\mathbb{C}}

\newcommand{\BZ}{\mathbb{Z}}
\newcommand{\BN}{\mathbb{N}}

\newcommand{\bu}{{\bf u}}

\newcommand{\bg}{{\bf g}}

\newcommand{\bF}{{\bf F}}

\newcommand{\bzero}{\mathbf 0}

\newcommand{\btheta}{\boldsymbol \theta}

\newcommand{\bomega}{\boldsymbol \omega}

\newcommand{\jpn}{\langle n\rangle}
\newcommand{\jpk}{\langle k\rangle}
\newcommand{\jpj}{\langle j\rangle}

\newcommand{\B}{\mathcal{B}}

\newcommand{\HT}{\mathcal{H}}

\newcommand{\lnorm}[1]{ \left\| #1 \right\|}

\newcommand{\commentK}[1]{\par\noindent\textcolor{blue}{\textbf{Kamran: \textit{#1}}}\par}

\newtheorem{theorem}{Theorem}[section]
\newtheorem{prop}{Proposition}[section]
\newtheorem{lemma}{Lemma}[section]
\newtheorem{cor}{Corollary}[section]

\newtheorem{remark}{Remark}[section]

\numberwithin{equation}{section}

\title[On the range of the planar $X$-ray transform]{On the range of the planar $X$-ray transform on the Fourier lattice of the torus}
\begin{document}
\date{\today}
\author{Kamran Sadiq}
\address{Faculty of Mathematics, University of Vienna, Oskar-Morgenstern-Platz 1, 1090 Vienna, Austria}
\email{kamran.sadiq@univie.ac.at}
\author{Alexandru Tamasan}
\address{Department of Mathematics, University of Central Florida, Orlando, 32816 Florida, USA}
\email{tamasan@math.ucf.edu}
\makeatletter
\@namedef{subjclassname@2020}{\textup{2020} Mathematics Subject Classification}
\makeatother
\maketitle

\begin{abstract}
We find necessary and sufficient conditions on the Fourier coefficients of a function $g$ on the torus to be in the range of the $X$-ray transform of functions with compact support in the plane, and establish the connection between the range characterization based on the Bukhgeim-Hilbert transform and the classical Gelfand-Graev, Helgason, and Ludwig characterization.
\end{abstract}

\section{Introduction}

We revisit the range characterization of the classical $X$-ray transform of a real valued function $f$ compactly supported in the plane. Since the $X$-ray and Radon transform \cite{radon1917} for planar functions differ merely by the way lines are parameterized, the necessary and sufficient constraints have been long established independently by Gelfand and Graev \cite{gelfandGraev}, Helgason  \cite{helgason65}, and Ludwig \cite{ludwig}. Due to their practical use in noise reduction \cite{yuWang07, yuWangEtall06}, completion of the data \cite{xia_etall, gompelDelfriseDyck, karp_etal, kudoSaito},  CT-hardware failure diagnosis \cite{patch},   the range characterization problem has been a continuing subject of research
\cite{chanLeng05,clackdoyle13, kazantsevBukhgeimJr04,kazantsevBukhgeimJr06, monard16}. 
Some recent work treat directly the discrete $X$-ray transform problem \cite{donsubPreprint}. 
Models which account for the attenuation have also been considered in the homogeneous case \cite{kuchmentLvin}, and in the non-homogeneous case in the breakthrough works \cite{ABK, novikov01,novikov02}, and subsequently \cite{natterer01, bomanStromberg, bal04,kazantsevBukhgeimJr07,sadiqTamasan01,monard17}. 
The references here are by no means exhaustive.

The corresponding problem for tensors of higher order in non-Euclidean spaces has also been considered in  \cite{vladimirBook,pestovUhlmann, AMU,venke20}, see \cite{paternainSaloUhlmann14} for a comprehensive review. 
On simple Riemannian surfaces, the range characterization of the geodesic $X$-ray transform of compactly supported functions has been established in terms of the scattering relation in the breakthrough work in \cite{pestovUhlmann}.  The connection between the Euclidean version of the characterization in \cite{pestovUhlmann} and the original characterization in \cite{gelfandGraev, helgason65, ludwig} was established in \cite{monard16}. 

In here we give yet another characterization of the range of the classical $X$-ray transform in terms of the Fourier coefficients of integrable functions on the torus, where the lines passing through the support of $f$ are parameterized by coordinates on a torus. Although $Xf$  is a function on the torus, our problem differs from the one in \cite{ilmavirta_etal}, where for a given direction (of rational slope) the integration takes place over a finite union of parallel segments in the unit disc. 

Apart from the symmetry constraints due to the double parameterization of the lines, of specific interest are the moment  conditions. 
The thrust of this work are the constraints \eqref{RTCond} replacing the moment conditions in \cite{gelfandGraev,helgason65,ludwig}, and the sufficiency part in Theorem \ref{RangeCharac}. The result is based on the authors' characterization in \cite{sadiqTamasan01} and uses a new mapping property (Theorem \ref{newmapping_Hilbert}) of the Bukhgeim-Hilbert transform corresponding to the $A$-analytic maps in the sense of \cite{bukhgeimBook}. 

All the details establishing notation and the statement of the main results are in Section \ref{sec:prelim}. In Section \ref{sec:L2map} we briefly recall existing results on $A$-analytic maps that are used in the proofs. In Section \ref{sec:newBHprop} we establish a new mapping property of the Bukhgeim-Hilbert transform, which is key to the proof of our main result in Section \ref{Sec:pf_mainTh}. In Section 6 we provide the missing connection between the original characterization in \cite{gelfandGraev, helgason65, ludwig} and the one in \cite{sadiqTamasan01}. To improve the readability of the work, some of the claims are proven in the appendix.


\section{Preliminaries and statement of main result}\label{sec:prelim}

Throughout,  $f$ is an integrable,  real valued  function, of compact support in the plane. Points $(x_1, x_2)$ in the plane are identified by the complex numbers $ x_1+ i x_2$, and directions $\btheta = (\cos \theta, \sin \theta)$ in the unit sphere $\sph$ by $e^{ i \theta}$. Upon a translation and scaling, $f$ is assumed supported in the unit disc $$\OM = \{ z  \in \BC : |z| <1  \}.$$ The boundary $\Gam$ of $\OM$ is the unit circle, but we keep this notation to differentiate from the set $\sph$ of directions. 

Lines $\ds L_{(\beta,\theta)}:=\{ e^{i\beta}+s e^{i\theta}: \; s\in\BR\}$ intersecting $\overline\OM$ are parametrized in coordinates $\ds \{(e^{i\beta},e^{i\theta}): \beta,\theta\in(-\pi,\pi]\}$, and
the $X$-ray transform of $f$\begin{align}\label{eq:xrayTransform}
Xf (e^{i \beta},e^{i \theta})= \int_{L_{(\beta,\theta)}}f ds 
\end{align}is understood as a function on the torus  $\Gam\times\sph$.

Since $\displaystyle L_{(\beta,\theta)}=L_{(2\theta -\beta-\pi,\theta)}= L_{(\beta,\theta+\pi)}=L_{(2\theta -\beta-\pi,\theta+\pi)},$ the set of lines intersecting $\overline\OM$ are quadruply covered when $(e^{i\beta},e^{i\theta})$ ranges over the entire torus $ \Gam\times \sph$. Moreover, the following symmetries are satisfied,
\begin{align}
&Xf (e^{i \beta},e^{i \theta})=Xf (e^{i(2\theta -\beta-\pi)},e^{i( \theta+\pi)}), \text{ and }
\label{Xray_symm*}\\
&Xf (e^{i \beta},e^{i \theta})=Xf (e^{i \beta},e^{i (\theta+\pi)}), \text{ for }  (e^{i \beta},e^{i \theta})\in\Gam\times\sph \label{X_ray_sym**};
\end{align}
see Figure \ref{fig:fanbeam1} below.

\begin{figure}[ht]
 \centering

  \pgfmathsetmacro{\Radius}{1.35} 
  \pgfmathsetmacro{\RRadius}{1.5*\Radius}
    \pgfmathsetmacro{\RRadiustheta}{0.67*\Radius}
  \pgfmathsetmacro{\Height}{1.25}
  \pgfmathsetmacro{\startAngle}{asin(\Height/\Radius)}
  
   \pgfmathsetmacro{\ptAngle}{\startAngle}
   \pgfmathsetmacro{\alphaAngle}{0.6*\startAngle}
   \pgfmathsetmacro{\thetaAngle}{\startAngle + \alphaAngle} 
   \pgfmathsetmacro{\gammaAngle}{\startAngle + 2*\alphaAngle-180} 
   \pgfmathsetmacro{\omegaAngle}{\thetaAngle - 90} 
   \pgfmathsetmacro{\xbpt}{cos(\ptAngle)*\Radius}
   \pgfmathsetmacro{\ybpt}{sin(\ptAngle)*\Radius} 
   
   \pgfmathsetmacro{\xpt}{cos(\ptAngle)*1.1*\RRadius}
   \pgfmathsetmacro{\ypt}{sin(\ptAngle)*1.1*\RRadius} 
   
   \pgfmathsetmacro{\Lxpt}{\xbpt + cos(\thetaAngle)}
   \pgfmathsetmacro{\Lypt}{\ybpt + sin(\thetaAngle)}
   
   \pgfmathsetmacro{\Lxxpt}{\xbpt - cos(\thetaAngle)*2.1*\Radius}
   \pgfmathsetmacro{\Lyypt}{\ybpt - sin(\thetaAngle)*2.1*\Radius}
   
   \pgfmathsetmacro{\omegaxpt}{cos(\omegaAngle)*\RRadius}
   \pgfmathsetmacro{\omegaypt}{sin(\omegaAngle)*\RRadius}

   \pgfmathsetmacro{\gammaxpt}{cos(\gammaAngle)*\Radius}
   \pgfmathsetmacro{\gammaypt}{sin(\gammaAngle)*\Radius}

\begin{tikzpicture}[scale=1.5,cap=round,>=latex]

\tikzset{
    thick/.style=      {line width=0.8pt},
    very thick/.style= {line width=1.1pt},
    ultra thick/.style={line width=1.6pt}
}

  \coordinate[label=below:$0$] (O) at (0,0,0);
  \filldraw[black] (O) circle(1.2pt);
  \draw[thick] (O) circle (\Radius cm);
  
  \coordinate[label=above:$\Omega$] (OM) at(120:0.85cm);
  
  \coordinate (E) at (\RRadius,0,0);
  \draw[dashed] (O) -- (E);

\coordinate (E1) at (\Radius,0,0);
\coordinate (E2) at (\RRadiustheta,0,0);
  
  \coordinate (P) at (\xbpt,\ybpt);
  \filldraw[black] (P) circle(1.2pt);
  
  \coordinate[label=right:$e^{ i \beta}$] (Z) at (\xbpt+0.025,\ybpt+0.1);

  \coordinate (P1) at (\gammaxpt,\gammaypt);
  \filldraw[black] (P1) circle(1.2pt);
  \draw[dashed] (O) -- (P1);

  \coordinate[label=above:$\nu$] (nu) at (\xpt,\ypt);
  \draw[thick,->] (O) -- (nu);
  
  \coordinate[label=right:$e^{i (2\theta -\beta-\pi)}$] (Z1) at (\gammaxpt+0.05,\gammaypt+0.05);

  \coordinate[label=above:$\btheta$] (L) at  (\Lxpt,\Lypt);
  \draw[thick,->] (P) -- (L);
  
  \coordinate (LL) at  (\Lxxpt,\Lyypt);
  \draw[thick,-] (P) -- (LL);
  
  
    \tikzset{
	position label/.style={
	  above = 3pt,
	  text height = 1.5ex,
	  text depth = 1ex
	},
      brace/.style={
	decoration={brace,mirror},
	decorate
      }
    }  
  \pic[draw,->,angle radius=.7cm,angle eccentricity=1.3,"$\beta$"] {angle=E--O--P}; 
    \pic[draw,->,angle radius=.55cm,angle eccentricity=1.29,"$\theta$"] {angle=E1--E2--P}; 
  \pic[draw,->,angle radius=.7cm,angle eccentricity=1.3,"$\alpha$"] {angle=nu--P--L}; 
  \pic[draw,->,angle radius=.7cm,angle eccentricity=1.3,"$\alpha$"] {angle=O--P--LL};
 \pic[draw,-,angle radius=.7cm,angle eccentricity=1.3,"$\alpha$"] {angle=P--P1--O};

\end{tikzpicture}
 \caption{Fan-beam coordinates: $e^{i \beta}\in\Gam$,  $e^{i\theta}\in\sph$, and $\btheta = (\cos \theta, \sin \theta)$.
} \label{fig:fanbeam1}
\end{figure}
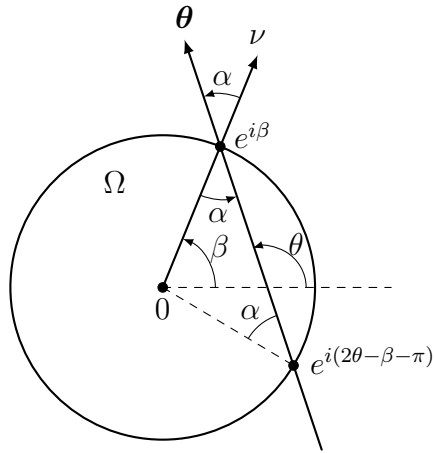

If $f$ is merely integrable in $\OM$, then $Xf$ may not be integrable on the torus. However, if  either \begin{align}\label{eq:f_refularity_cond1}
	\textnormal{supp }f  \subset \OM, 
	\mbox{ or }  f \in L^p(\OM)\mbox{ for some }p >2,
	\end{align}
	then $Xf \in L^1(\Gamma\times\sph)$; see Proposition \ref{prop:fLp_regularity} in the appendix.
	
We consider the partition of the torus into three parts: the ``outflux" part
\begin{equation}\label{eq:Gam+}
\Gam_+:=
\left \{(e^{i \beta} ,e^{i(\beta+\alpha )} )\in \Gam \times\sph:\,  \beta\in (-\pi, \pi] , \; |\alpha|<\frac{\pi}{2} \right \},
\end{equation}the  ``influx" part
\begin{equation}\label{eq:Gam-}
\Gam_-:=
\left \{(e^{i \beta} ,e^{i(\beta+\alpha )} )\in \Gam \times\sph:\,  \beta\in (-\pi, \pi] , \; \frac{\pi}{2} < |\alpha|\leq\pi  \right \},
\end{equation}and the (Lebesgue negligible) variety $\Gam_0:=(\Gam\times\sph)\setminus(\Gam_+\cup\Gam_-)$ parameterizing the tangent lines to the circle; see Figure \ref{fig:fanbeam1}.



Motivated by \eqref{Xray_symm*}, let $L^1_{sym}(\OM\times\sph)$ denote the space of integrable functions $g$ on the torus satisfying the symmetry relation
\begin{align}\label{sym*}
g (e^{i \beta},e^{i \theta})=g (e^{i(2\theta -\beta-\pi)},e^{i( \theta+\pi)}),\text{ for a.e. } (e^{i \beta},e^{i \theta})\in\Gam\times\sph.
\end{align}
Since 
$\left(e^{i\beta},e^{i \theta}\right)$ and $\left(e^{i(2\theta-\beta-\pi)},e^{( \theta+\pi)}\right)$ are either both in $\Gam_+$, or both in $\Gam_-$, we can consider  the spaces  $L^1_{sym}(\Gam_\pm)$ of integrable functions on the half-tori  satisfying \eqref{sym*}. Clearly,  $\ds g\in L^1_{sym}(\Gam\times\sph)$ if and only if its restrictions $\ds g|_{\Gam_\pm}\in L^1_{sym}(\Gam_\pm)$. The symmetry relation \eqref{Xray_symm*} yields $Xf\in L^1_{sym}(\Gam\times\sph)$. 

Furthermore, we consider the subspace $L^1_{sym,odd}(\Gam\times\sph)$ of functions $g\in L^1_{sym}(\Gam\times\sph)$, which, in addition to satisfying \eqref{sym*}, they are also odd with respect to the angular variable:
\begin{align}
g (e^{i \beta},e^{i \theta})=- g(e^{i \beta},e^{i (\theta+\pi)})\label{sym**}.
\end{align}Note that $Xf\not\in L^1_{sym,odd}(\Gam\times\sph)$, since the symmetry relation  \eqref{X_ray_sym**} is broken.

Our main result gives necessary and sufficient conditions for a function $g\in L^1_{sym,odd}(\Gam\times\sph)$ to satisfy
\begin{equation}
g (e^{i \beta},e^{i \theta})= Xf(e^{i \beta},e^{i \theta}), \text{ a.e. }(e^{i \beta},e^{i \theta})\in\Gam_+,
\end{equation}for some real valued $f\in L^1(\OM)$. Note that \eqref{sym**} implies $g= - Xf$ on $\Gam_-$. 

The characterization is in terms of the Fourier coefficients
\begin{align}\label{eq:gnk}
g_{n,k} := \frac{1}{(2 \pi)^2} \int_{-\pi}^{\pi} \int_{-\pi}^{\pi} g(e^{i \beta}, e^{i\theta}) e^{-i n \theta}  e^{-i k \beta} d \theta d \beta,\; n,k \in \BZ,
\end{align}of $g$ on the lattice $\BZ \times \BZ$. The order of the indexes play a role. Throughout, the first index is the Fourier mode in the angular variable on $\sph$, and we call it an \emph{angular mode}. The second index is the mode in the boundary variable on $\Gam$, and we call it a \emph{boundary mode}. 

In the statements below we use the notations $C^\mu(\OM)$, $0<\mu<1$, for the space of locally H\"older continuous functions,  and  $\jpn=(1+|n|^2)^{1/2}$.


\begin{figure}[hb!]\label{lattice}
	\centering
	\begin{tikzpicture}[scale=0.7,cap=round,>=latex]
	\tikzset{help lines/.style={color=blue!50}}
	\draw[thick,step=1cm,help lines] (0,0) grid (9,9);
	\draw[ultra thin,step=1cm,help lines] (0,0) grid (9,9);
	
	\draw[dashed,step=1cm,help lines] (9,3) to (3,9);
	\draw[thin] (9,2.5) to (0,7);
	
	\draw[ultra thick,-latex] (9.2,3) -- (-0.7,3) node[anchor=east] {$-n$};
	\draw[ultra thick,-latex] (9,0) -- (9,10) node[anchor=east] {$k$};
	\foreach \x in {0,1,...,9} {     
		\draw [thick] (\x,0) -- (\x,-0.2);
	}
	\foreach \y in {0,1,...,9} {   
		\draw [thick] (9,\y) -- (9.1,\y);
	}
	\foreach \x in {0,1,...,4}
	{
		\pgfmathtruncatemacro{\rx}{-9+2*\x};
		\node [anchor=north] at (2*\x,2.95)  {\rx};
		\draw [ultra thick] (2*\x,3.1) -- (2*\x,2.9);
	} 
	
	\foreach \y in {0,1,...,9}
	{
		\pgfmathtruncatemacro{\rr}{ -3 + \y}; 
		\node [anchor=east] at (9.8,\y)  {\rr};
	} 
	
	\coordinate[label=below:$n+2k \leq -1$] (origin) at (2.9,4.9);

	\draw[black] (8,4) circle (5pt);  
	\draw[black] (6,6) circle (5pt);  
	\draw[black] (4,8) circle (5pt);  
	
	\node at (8,5)[draw,rectangle,minimum height=0.1cm]{}; 
	\node at (6,5)[draw,rectangle,minimum height=0.1cm]{}; 

	\node[isosceles triangle,
	isosceles triangle apex angle=60,
	draw,
	rotate=270,  
	minimum size =0.01cm] (T1)at (8,6){};

	\node[isosceles triangle,
	isosceles triangle apex angle=60,
	draw,
	rotate=270,  
	minimum size =0.01cm] (T2)at (4,6){};
	
	\node at (8,7)[draw,trapezium,minimum height=0.1cm]{}; 
	\node at (2,7)[draw,trapezium,minimum height=0.1cm]{}; 
	
	\draw[black] (6,7) ellipse (3pt and 6pt);
	\draw[black] (4,7) ellipse (3pt and 6pt);
	
	
	\node [rectangle, draw, xslant=0.4] at (8,8) {};
	\node [rectangle, draw, xslant=0.4] at (0,8) {};
	
	\node[regular polygon,
	draw,
	regular polygon sides = 5] (p) at (6,8) {};
	
	\node[regular polygon,
	draw,
	regular polygon sides = 5] (p) at (2,8) {};

	\foreach \x in {.5,1.5,...,8.5} {
	}
	\foreach \y in {.5,1.5,...,8.5} {
	}
	\end{tikzpicture}
	\caption{$f$ is determined by the odd negative angular modes on or above the diagonal $k=-n$. The diagonal modes $g_{n,-n}$ are real valued. All the odd non-positive angular modes on and below the line $n+2k=-1$ vanish.}
\end{figure}
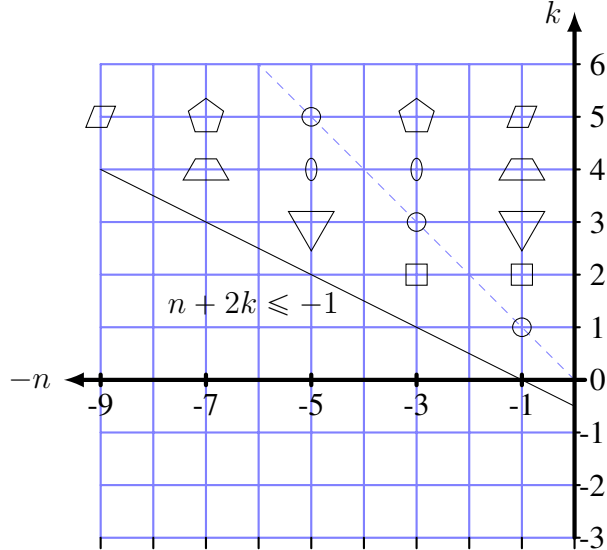


\begin{theorem}[Range characterization]\label{RangeCharac}

(i)  Let $ f \in L^1(\OM)$ be real valued satisfying \eqref{eq:f_refularity_cond1}, and $g\in L^1_{sym,odd}(\Gam\times\sph)$, with $$g= Xf \text{ on }\Gam_+.$$
Then  the Fourier coefficients $\{g_{n,k}\}_{n, k \in \BZ}$, of $g$ satisfy the following conditions.
\begin{alignat}{3}
&\text{Oddness}: && \quad g_{n,k}=0, && \quad \text{for all } n\in \BZ \text{ even, and all } k\in\BZ; \label{RTCond_odd}\\
&\text{Conjugacy}: && \quad g_{-n,-k}= \ol{g_{n,k}},  &&\quad \text{for all} \;  n, k \in \BZ;\label{RTCond_reality}\\
&\text{Symmetry}:
&& \quad g_{n,k}=(-1)^{n+k}g_{n+2k,-k}, &&\quad \text{for all} \; n,k\in\BZ; \label{RT_FourierEven}\\
&\text{Moments}: &&\quad  g_{n,k}=  (-1)^k  g_{n+2k,-k}, &&\quad  \text{for all } n\leq -1 \text{ odd,  and all } k\leq 0.\label{RTCond}
\end{alignat}

	(ii) Let $\{g_{n,k}\}$ be given for all $n \leq -1$ odd, and $k \in\BZ$ such that 
\begin{align}\label{gnk_decay}
	\sum_{\substack{n\leq -1\\
			n =  \,\text{odd}}} \jpn^{2} \sum_{k=-\INF }^\INF \lvert g_{n,k} \rvert < \INF, \quad \text{and} \quad
	\sum_{k=-\INF}^\INF  \jpk^{1+\mu}
	\sum_{\substack{n\leq -1\\
			n  =\,\text{odd}}} \lvert g_{n,k} \rvert < \INF, 
\end{align}for some $\mu>1/2$.

If $\{g_{n,k}\}$ satisfy \eqref{RT_FourierEven} and \eqref{RTCond}, then there exists a real valued $f \in L^1(\OM) \cap C^\mu(\OM)$ such that the mapping 
\begin{align}\label{eq:gnk_map}
	(\Gamma\times\sph) \ni (e^{i \beta}, e^{i\theta}) \longmapsto  
	2\re\left\{\sum_{\substack{n\leq-1\\
			n  = \,\text{odd}}}  \sum_{k \in\BZ} g_{n,k} e^{i n \theta}  e^{i k \beta}\right\}
\end{align} 
defines a function in $L^1_{sym,odd}(\Gam\times\sph)$, which coincides with $Xf$ on $\Gam_+$ (and with $- Xf$ on $\Gam_-$).
\end{theorem}

The thrust of this work are the constraints \eqref{RTCond} replacing the moment conditions 
(see Remark \ref {GGHL=RTCond}), and the sufficiency part in Theorem \ref{RangeCharac} for functions of finite smoothness. In particular, for $n$ odd, the right hand sides of \eqref{RTCond} and \eqref{RT_FourierEven} differ by a sign. As a direct  consequence, the following holds.

\begin{cor}\label{corollary}
(i) Let $g\in L^1_{sym,odd}(\Gam\times\sph)$ coincide with $Xf$ on $\Gam_+$, and $\{g_{n,k}\}$ be its Fourier coefficients.  Then, for all $n\leq -1$ odd,
\begin{equation}\label{RTCond2neg}
g_{n,k}=\left\{
\begin{array}{ll}
0,& \text{ if }
k\leq -\frac{n+1}{2},\\
(-1)^{k+1}\overline{g_{-n-2k,k}},& \text{ if }
k\geq \frac{-n+1}{2},
\end{array}
\right.
\end{equation}see Figure \ref{lattice}.

(ii) Let $\{g_{n,k}\}$ be given for $n \leq -1$ odd, and $\ds k \geq\frac{-n-1}{2} $ such that 
\begin{align}\label{gnk_decay2}
	\sum_{\substack{n\leq -1\\
			n =  \,\text{odd}}} \jpn^{2} \sum_{k=\frac{-n-1}{2}}^\INF \lvert g_{n,k} \rvert < \INF, \quad \text{and} \quad
	\sum_{k=1}^\INF \jpk^{1+\mu}\sum_{\substack{n=-2k-1\\
			n  =\,\text{odd}}}^{-1}
			\lvert g_{n,k} \rvert < \INF, 
\end{align}for some $\mu>1/2$. If $\{g_{n,k}\}$ satisfies \eqref{RTCond2neg}, then there exists a real valued function $f \in L^1(\OM) \cap C^\mu(\OM)$, such that the mapping 
\begin{align}\label{eq:gnk_map}
	(\Gamma\times\sph) \ni (e^{i \beta}, e^{i\theta}) \longmapsto  2\re\left\{\sum_{\substack{n\leq-1\\
			n  = \,\text{odd}}}  \sum_{k=\frac{-n-1}{2}}^\INF g_{n,k} e^{i n \theta}  e^{i k \beta}\right\}
\end{align} 
is precisely $Xf$ on $\Gam_+$, and $-Xf$ on $\Gam_-$.
\end{cor}

The oddness and conjugacy constraints in \eqref{RTCond_odd} and \eqref{RTCond_reality} are not intrinsic to the $X$-ray transform. The symmetry constraints \eqref{RT_FourierEven} merely account for each line being doubly parametrized in $\Gam_+$, and they are shared by any function on the torus satisfying the symmetry \eqref{sym*}; see Lemma \ref {lem:evenness} in the appendix. 


For a complex valued $f$, since $\ds \re (Xf ) = X(\re(f))$ and $\ds \im (Xf) = X(\im(f))$,  Theorem \ref{RangeCharac} and Corollary \ref{corollary} apply  to  $\re (f)$ and to $\im(f)$, respectively.

\section{$L^2$-analytic maps and their trace characterization} \label{sec:L2map}
The method of proof of \eqref{RTCond} is based on the characterization in \cite{sadiqTamasan01}
of traces of $A$-analytic maps in the sense of Bukhgeim \cite{bukhgeimBook}. In this section we summarize those existing results used in the proof of Theorem \ref{RangeCharac}. 
These results hold for $\OM$ a strictly convex domain, not necessarily the unit disk.

We approach the range characterization via the well-known connection with the transport model, where the unique solution $u(z,\btheta)$ to the boundary value problem
\begin{subequations}\label{bvp_transport}
\begin{align}\label{TransportEq1}
\btheta\cdot\nabla u(z,\btheta) &=  2f(z) , \quad (z,\btheta)\in \OM\times\sph, \\  \label{u_Gam-}
u \lvert_{\Gam_{-}} &= - Xf\lvert_{\Gam_-}
\end{align}
\end{subequations}has the trace $u\lvert_{\Gam\times\sph}$ in $L^1_{sym,odd}(\Gam\times\sph)$, and $u|_{\Gam_+}= Xf|_{\Gam_+}$ on $\Gam_+$.

The advection operator $\btheta \cdot\nabla$ in complex notation becomes $e^{-i\tta}\dba + e^{i\tta}\del$, 
where $\btheta=(\cos\tta,\sin\tta)$, and  $\dba = (\del_{x_1}+i\del_{x_2})/2$ and $\del =(\del_{x_1}-i\del_{x_2})/2$ are the Cauchy-Riemann operators.


If $\ds \sum_{n \in \BZ} u_{n}(z) e^{in\tta}$ is the Fourier series  expansion in the angular variable $\btheta$ of a solution $u$ of \eqref{TransportEq1}, then, provided some sufficient decay (to be specified later) of $u_n$ to allow regrouping, 
\begin{equation} \label{eq:transport_fourier}
 \begin{aligned}
  \btheta \cdot\nabla u(z,\btheta) 
 &= \dba u_{1}(z) + \del u_{-1}(z) + \underset{{\substack{n \in \BZ \\ n\neq 1 } }}{\sum}  \left( \dba u_{n}(z)+ \del u_{n-2}(z)  \right) e^{i(n-1)\tta}.
 \end{aligned}
\end{equation}
By identifying the Fourier modes of the same order, the equation \eqref{TransportEq1} reduces to the system:
\begin{align}\label{freeeq}
\overline{\del} u_{1}(z)+\del u_{-1}(z) =2f(z), 
\end{align}
and 
\begin{align}\label{infsys0}
\dba u_n(z) +\del u_{n-2}(z) =0,\quad n\neq 1.
\end{align}
For a real valued $u(z,\btheta)$, $u_{-n}=\ol{u_n}$ and the angular dependence is completely determined by the sequence $\bu$ of its nonpositive Fourier modes,
\begin{align}\label{boldu}
\OM \ni z\mapsto  \bu(z)&: = \langle u_{0}(z), u_{-1}(z),u_{-2}(z),... \rangle.
\end{align}
In particular, $\bu$ solves the Beltrami-like equation
\begin{align}\label{beltrami}
\dba\bu(z) +L^2 \del\bu(z) = 0,\quad z\in \OM,
\end{align}
where  $L\bu(z)=L (u_0(z),u_{-1}(z),u_{-2}(z),...):=(u_{-1}(z),u_{-2}(z),...)$ denotes the left translation.


Bukhgeim's original  theory in \cite{bukhgeimBook}  shows that solutions of \eqref{beltrami}, called $L^2$-analytic, satisfy a Cauchy-like integral formula,
 \begin{align}\label{Analytic}
\bu (z) = \B [\bu \lvert_{\Gam}](z), \quad  z\in\OM,
\end{align} where $\B$ is 
 the Bukhgeim-Cauchy operator  acting on $\bu \lvert_{\Gam}$. We use the formula in \cite{finch}, where
  $\B$ is defined component-wise for $n\geq 0$ by
\begin{align} \label{BukhgeimCauchyFormula}
 (\B \bu)_{-n}(z) &:= \frac{1}{2\pi i} \int_{\Gam}
\frac{ u_{-n}(\zeta)}{\zeta-z}d\zeta  + \frac{1}{2\pi i}\int_{\Gam} \left \{ \frac{d\zeta}{\zeta-z}-\frac{d \ol{\zeta}}{\ol{\zeta}-\ol{z}} \right \} \sum_{j=1}^{\infty}  
 u_{-n-2j}(\zeta)
\left( \frac{\ol{\zeta}-\ol{z}}{\zeta-z} \right) ^{j},\; z\in\OM.
\end{align}


Similar to the analytic maps, the traces of $L^2$-analytic maps  on the boundary must satisfy some constraints, which can be expressed in terms of a corresponding Hilbert-like transform introduced in  \cite{sadiqTamasan01}. More precisely, the Bukhgeim-Hilbert transform $\HT$ acting on  $\bg$, 
\begin{align}\label{boldHg}
\Gam \ni z\mapsto  (\HT \bg)(z)& = \langle (\HT \bg)_{0}(z), (\HT \bg)_{-1}(z),(\HT \bg)_{-2}(z),... \rangle
\end{align}
is defined component-wise for $n\geq 0$ by
\begin{align}\label{BHtransform}
(\HT \bg)_{-n}(z)=\frac{1}{\pi }\int_\Gam \frac{ g_{-n}(\zeta)}{\zeta-z}d\zeta  + \frac{1}{\pi }\int_{\Gam} \left \{ \frac{d\zeta}{\zeta-z}-\frac{d \ol{\zeta}}{\ol{\zeta-z}} \right \} \sum_{j=1}^{\infty}  
g_{-n-2j}(\zeta)
\left( \frac{\ol{\zeta-z}}{\zeta-z} \right) ^{j}, z\in\Gam.
\end{align}

The theorems below comprise some results in  \cite{sadiqTamasan01,sadiqTamasan02}. For $0<\mu<1$, $p=1,2$, we consider the  Banach spaces:
\begin{equation}\label{spaces}
 \begin{aligned} 
 l^{1,p}_{\INF}(\Gam) &:= \left \{ \bg= \langle g_{0}, g_{-1}, g_{-2},...\rangle\; : \lnorm{\bg}_{l^{1,p}_{\INF}(\Gam)}:= \sup_{\xi \in \Gam}\sum_{j=0}^{\INF}  \jpj^p \lvert g_{-j}(\xi) \rvert < \INF \right \},\\
 C^{\mu}(\Gam; l_1) &:= \left \{ \bg= \langle g_{0}, g_{-1}, g_{-2},...\rangle:
\sup_{\xi\in \Gam} \lVert \bg(\xi)\rVert_{\ds l_{1}} + \underset{{\substack{
            \xi,\eta \in \Gam \\
            \xi\neq \eta } }}{\sup}
 \frac{\lVert \bg(\xi) - \bg(\eta)\rVert_{\ds l_{1}}}{|\xi - \eta|^{ \mu}} < \INF \right \}, \\
  Y_{\mu}(\Gam) &:= \left \{ \bg: \bg \in  l^{1,2}_{\INF}(\Gam) \; \text{and} \;
           \underset{{\substack{
             \xi,\eta \in \Gam \\
             \xi\neq \eta } }}{\sup} \sum_{j=0}^{\INF}  \jpj 
  \frac{\lvert g_{-j}(\xi) - g_{-j}(\eta)\rvert }{|\xi - \eta|^{ \mu}} < \INF \right \},
 \end{aligned}
 \end{equation} where, for brevity, we use the notation $\jpj=(1+|j|^2)^{1/2}$.
Similarly,  we consider $ C^{\mu}(\ol \OM; l_\INF) $, and $  C^{\mu}(\OM; l_\INF) = \bigcup_{0<r<1} C^{\mu}(\ol \OM_r; l_\INF)$, where for $0<r<1$, $\OM_r = \{ z \in \BC : |z| <r  \}$.

\begin{theorem}\label{BukhgeimCauchyThm}
Let $0<\mu<1$. Let $\bg = \langle g_{0}, g_{-1}, g_{-2},...\rangle$ be a sequence valued map defined on the boundary $\Gam$ and $\B$ be the Bukhgeim-Cauchy operator acting on $\bg$ as in \eqref{BukhgeimCauchyFormula}. 

(i) If $\bg \in l^{1,1}_{\INF}(\Gam)\cap C^\mu(\Gam;l_1)$, then $\bu := \B \bg\in C^{1,\mu}(\OM;l_\infty)\cap C(\ol \OM;l_\infty)$ is $L^2$-analytic in $\OM$.

(ii) Moreover, if $\bg\in Y_{\mu}(\Gam)$ for $\mu>1/2$, then $ \B \bg \in C^{1,\mu}(\OM;l_1)\cap C^{\mu}(\ol \OM;l_1)$. 
\end{theorem}For the proof of Theorem \ref {BukhgeimCauchyThm} (i) we refer to \cite[Theorem 3.1]{sadiqTamasan01}, and for part (ii) we refer to \cite[Proposition 2.3]{sadiqTamasan02}.


Key to the proof of the constraints \eqref{RTCond} is the following characterization of traces of $L^2$-analytic maps.
\begin{theorem}\label{NecSuf_BukhgeimHilbert_Thm}

Let $0<\mu<1$, and let $\HT$ be the Bukhgeim-Hilbert transform in \eqref{BHtransform}.
%
%

(i) If $\bg \in l^{1,1}_{\INF}(\Gam)\cap C^\mu(\Gam;l_1)$ is the boundary value of an $L^2$-analytic function,
then $\HT \bg\in C^{\mu}(\Gam;l_\infty)$ and satisfies 
\begin{align} \label{NecSufEq}
(I+ i \HT) \bg = {\bf {0}}.
\end{align}
(ii)  If $\bg\in Y_{\mu}(\Gam)$ for $\mu>1/2$,  satisfies \eqref{NecSufEq}, then there exists an $L^2$-analytic function
$ \bu \in C^{1,\mu}(\OM;l_1)\cap C^{\mu}(\ol \OM;l_1)$, such that
\begin{align}\label{gdata_defn}
  \bu \lvert_{\Gam} = \bg.
\end{align}
%
\end{theorem}
For the proof  of Theorem \ref{NecSuf_BukhgeimHilbert_Thm}  we refer to \cite[Proposition 3.1, Theorem 3.2,  Corollary 4.1, and Proposition 4.2]{sadiqTamasan01}.



\section{Properties of the Bukhgeim-Hilbert transform on the Fourier lattice}\label{sec:newBHprop}


In this section, we present a new mapping property of the Bukhgeim-Hilbert transform. In here $\Omega$ is the unit disk, and $\Gam$ is its unit circle boundary.

Given $\bg = \langle g_0, g_{-1}, g_{-2}, ... \rangle \in l_{\INF}(\BN; L^1(\Gam))$, we consider the Fourier coefficients of its components 
\begin{align}\label{gnk_bg}
g_{-n,k} := \frac{1}{2 \pi} \int_{-\pi}^{ \pi} g_{-n} \left( e^{i  \beta} \right) e^{-i k \beta} d \beta, \text{ for all } n \geq 0, \text{ and } k \in \BZ.
\end{align}


\begin{theorem}\label{newmapping_Hilbert}
Let $\bg=  \langle g_{0},g_{-1}, g_{-2} ... \rangle \in l^{1,1}_{\INF}(\Gam)\cap C^\mu(\Gam;l_1)$, $0< \mu <1$, and $g_{-n,k}$ be the Fourier coefficients of its components as in \eqref{gnk_bg}.
Let $\ds \HT \bg = \langle (\HT \bg)_{0}, (\HT \bg)_{-1},(\HT \bg)_{-2},... \rangle$ be the Bukhgeim-Hilbert transform  acting on $\bg$ as defined in \eqref{BHtransform}.
Then  $\HT \bg\in C^{\mu}(\Gam;l_\infty)$, and the Fourier coefficients 
$\ds (\HT \bg)_{-n,k} := \frac{1}{2 \pi} \int_{0}^{2 \pi} (\HT \bg)_{-n} \left( e^{i  \beta} \right) e^{-i k \beta} d \beta$, for $ n \geq 0,  k \in \BZ, $
of its components satisfy 
\begin{align}\label{Hu_comp}
(-i) (\HT \bg)_{-n,k}=\left\{
\begin{array}{ll}g_{-n,k}&\mbox{ if } k\geq 0,\\
-g_{-n,k}+ 2(-1)^k g_{-n+2k,-k}&\mbox{ if } k\leq -1.
\end{array}
\right.\end{align}
\end{theorem}

\begin{proof} 

Since $\bg=  \langle g_{0},g_{-1}, g_{-2} ... \rangle\in l^{1,1}_{\INF}(\Gam)\cap C^\mu(\Gam;l_1)$, by Theorem \ref{NecSuf_BukhgeimHilbert_Thm} (i), $\HT \bg\in C^{\mu}(\Gam;l_\infty) $.
 
 For $\xi = e^{i \beta}$ and $\zeta = e^{i \alpha} $ on $\Gam$, 
 \begin{align*}
 \frac{d\zeta}{\zeta-\xi} = \frac{i e^{i\alpha}d \alpha}{e^{i\alpha}-e^{i\beta }}, \quad 
 \frac{d\zeta}{\zeta-\xi}-\frac{d \ol{\zeta}}{\ol{\zeta}-\ol{\xi}} = id \alpha, \quad 
 \left( \frac{\ol{\zeta}-\ol{\xi}}{\zeta-\xi} \right) = - e^{-i (\beta+\alpha)},
 \end{align*}
 and the components $\ds  (\HT \bg)_{n}$, $n \geq 0$,  rewrite as 
 \begin{align}\label{eq:HT_fanbeam}
 (-i)\ds(\HT\bg)_{-n}(e^{i\beta})= 
 \frac{1}{ \pi} \int_{-\pi}^{ \pi}   \frac{g_{-n}(e^{i\alpha})}{e^{i\alpha}-e^{i\beta }} e^{i\alpha} d \alpha +
  \frac{1}{\pi} \sum_{k=1}^{\infty} (-1)^k e^{-i \beta k } \int_{-\pi}^{ \pi}  g_{-n-2k}(e^{i\alpha}) e^{-i \alpha k } d \alpha.
 \end{align}

Since $\ds g_{-n}(e^{i\alpha})=\sum_{k=-\infty}^\infty g_{-n,k}e^{ik\alpha},$ by using Lemma \ref{singularintegral} (in the appendix),  the first term on the right-hand-side of \eqref{eq:HT_fanbeam} becomes
\begin{align*}
 \frac{1}{\pi} \int_{-\pi}^{ \pi}   \frac{g_{-n}(e^{i\alpha})}{e^{i\alpha}-e^{i\beta }} e^{i\alpha} d \alpha &= \sum_{k=-\infty}^{\infty} g_{-n,k} \frac{1}{\pi} \int_{-\pi}^{\pi} \frac{e^{i(k+1)\alpha}}{e^{i\alpha}-e^{i\beta}} d \alpha = \sum_{k=-\infty}^{\infty} \sgn{(k)} g_{-n,k} e^{i k \beta}.
\end{align*} 
The last term of \eqref{eq:HT_fanbeam},
\begin{align*}
\frac{1}{\pi} \sum_{k=1}^{\infty} (-1)^k e^{-i \beta k } \int_{-\pi}^{ \pi}  g_{-n-2k}(e^{i\alpha}) e^{-i \alpha k } d \alpha 
&= \sum_{k=1}^{\infty}2 (-1)^k g_{-n-2k,k}e^{-i \beta k }.
\end{align*}
Thus, the right-hand-side of \eqref{eq:HT_fanbeam} is
\begin{align*}
\sum_{k=-\infty}^{\infty} \sgn{(k)} g_{-n,k} e^{i k \beta} +
\sum_{k=-\infty}^{-1}2 (-1)^k g_{-n+2k,-k}e^{i \beta k },
\end{align*}
and the Fourier coefficients $ (\HT \bg)_{-n,k}$ satisfy \eqref{Hu_comp}.
\end{proof}

Using the Fourier representation \eqref{Hu_comp} of the Bukhgeim-Hilbert transform (applied twice), it is now easy to see that $\HT$  enjoys the idempotent property.

\begin{cor}\label{newmapping_Hilbert1}
	Let $0< \mu <1$. If  $\bg \in l^{1,1}_{\INF}(\Gam)\cap C^\mu(\Gam;l_1)$, then $\ds \HT^2 \bg  =  - \bg$.
\end{cor}

\section{ Proof of Theorem \ref{RangeCharac}}\label{Sec:pf_mainTh}

(i) {\bf Necessity:} Since $g$ is angularly odd, \eqref{RTCond_odd} holds.
Since $g$ is real valued, \eqref{RTCond_reality} holds. The identities \eqref {RT_FourierEven} follow by direct calculation, see Lemma \ref{lem:evenness}.

Recall that 
\begin{equation}\label{eq:gstar1}
g=
\left\{
\begin{array}{ll}
Xf ,&\mbox{ on } \Gam_+,\\
-Xf, &\mbox{ on } \Gam_-
\end{array}
\right.
\end{equation}is the trace on $\Gam \times \sph$ of the unique solution
of the boundary value problem 
 
\begin{subequations}\label{bvp_transport_odd}
\begin{align}\label{TransportEqOdd}
\btheta\cdot\nabla u(z,\btheta) &=  2f(z) , \quad (z,\btheta)\in \OM\times\sph, \\  \label{u_Gam-}
u\lvert_{\Gam_{-}} &= - Xf\lvert_{\Gam_-}.
\end{align}
\end{subequations}
Moreover, $u$ must be angularly odd, $u(z,\btheta)=-u(z,-\btheta)$: Indeed, if $u^{odd}(z,\btheta):=\frac{1}{2}\left[u(z,\btheta)-u(z,-\btheta)\right]$ denotes the angularly odd part of $u$, then 
\begin{align*}
\btheta\cdot\nabla u^{odd}(z,\btheta) = \frac{1}{2}\left[\btheta\cdot\nabla u(z,\btheta)+(-\btheta)\cdot\nabla u(z,-\btheta) \right]=2f(z),
\end{align*}and $u^{even}:=u -u^{odd}$ solves
\begin{subequations}\label{bvp_transport_even}
\begin{align}
\btheta\cdot\nabla u^{even}(z,\btheta) &=  0 , \quad (z,\btheta)\in \OM\times\sph, \\  
u^{even}\lvert_{\Gam_{-}} &=0.
\end{align}
\end{subequations}
Since \eqref{bvp_transport_even} has the unique solution $u^{even}\equiv 0$, $u=u^{odd}$ in $\ol\OM\times\sph$.

We will first prove \eqref{RTCond} for $ f \in C^{2}_0(\OM)$ real valued of compact support in $\OM$. The result for $f \in L^1(\OM)$ follows by a density argument.

Since $f$ is real valued, the solution $\ds u(z,\btheta) $ of \eqref{bvp_transport_odd} is also  real valued, and its Fourier modes in the angular variable occur in conjugates, $u_{-n}(z)=\ol{u_n} (z), \, n\in \BZ$. 
Let $\bu$ be  the sequence valued map of the non-positive odd Fourier modes
\begin{align}\label{boldu1}
\OM \ni z\mapsto  \bu(z)&: = \langle u_{-1}(z), u_{-3}(z),u_{-5}(z),... \rangle.
\end{align}
Since $ f \in C^{2}_0(\OM)$,  $ u \in C^{2}(\ol \OM\times \sph)$, in particular $\bu \in C^1(\ol \OM; l_1)$. 

Since $ u \in C^{2}(\ol \OM\times \sph)$, its trace $g=u \lvert_{\Gam \times \sph}\in C^2(\Gam\times\sph)$. We define the sequence valued map $\bg$ on the boundary $\Gam$ by
\begin{align}\label{eq:bg1}
\Gam \ni e^{i \beta} \mapsto  \bg(e^{i \beta}) := \langle g_{-1} (e^{i \beta}), g_{-3}(e^{i \beta}), g_{-5} (e^{i \beta}), ... \rangle, 
\end{align}
where $\ds g_{-n}(e^{i \beta}) = \frac{1}{2 \pi} \int_{-\pi}^{ \pi} g(e^{i \beta}, e^{i\theta}) e^{i n \theta} d\theta$ is the $(-n)$-th Fourier coefficients in the angular variable of the function $g$.


 Since $g=u \lvert_{\Gam \times \sph}$, $\bg = \bu \lvert_{\Gam} \, \in C^1(\Gam; l_1) \subset  l^{1,1}_{\INF}(\Gam) \cap C^{\mu}(\Gam; l^1)$, $\mu >1/2$.

By \eqref{beltrami}, $\bu$ is $L^2$-analytic in $\OM$, and, then, $\bg$ is the boundary value of an $L^2$-analytic map.
 
 By the necessity part in Theorem \ref{NecSuf_BukhgeimHilbert_Thm}, we have $\HT \bg \in C^{\mu}(\Gam; l_\INF)$ and
\begin{align}\label{Hg_char}
(I+i\HT) \bg= \bzero,
\end{align} where $\HT$ is the Bukhgeim-Hilbert operator in \eqref{BHtransform}.
In particular, 
\begin{align}\label{Hgnk_char}
([I+i\HT] \bg)_{n,k}= 0,
\end{align}
for all $n \leq -1$ odd, and $k \in \BZ$.

By Theorem \ref{newmapping_Hilbert}, the Fourier coefficients $\ds (\HT \bg)_{n,k}$, for  $ n \leq -1$ odd, $k \in \BZ$,
satisfy \eqref{Hu_comp}, 
and thus
\begin{align}\label{eq:HT_gnk}
([I+i\HT] \bg)_{n,k}=\left\{
\begin{array}{ll} 0 &\mbox{ if } k\geq 0,\\
2g_{n,k}-2 (-1)^k  g _{n+2k,-k}&\mbox{ if } k\leq -1.
\end{array}
\right.\end{align}

From  \eqref{Hgnk_char}, the Fourier coefficients of $\bg$ must satisfy \eqref{RTCond}:
\begin{align}
	g_{n,k}=  (-1)^k  g_{n+2k,-k}, \text{ for all  } n\leq -1 \text{ odd, } \text{ and } k\leq -1.
\end{align}
Equation \eqref{RTCond} for $k=0$ is trivially satisfied.
 
The proof for $f \in L^1(\OM)$ follows from the density of $C^2_0( \OM)$ in $L^1(\OM)$.

\vspace{1cm}

(ii) {\bf Sufficiency:} Given the double sequence $\{g_{n,k}\}$ for $n \leq -1$ odd, and $k\in \BZ$, we construct a real valued function $f$ in $\OM$ such that  the map on the torus
$\ds \left\{
\begin{array}{ll}
Xf \text{ on }\Gam_+,\\
-Xf \text{ on }\Gam_-\\
\end{array}\right.$
has its  Fourier coefficients  equal to $\{g_{n,k}\}$.

Define first the sequence valued map on the boundary $\Gam$ by
\begin{align}\label{bg_gnk}
\Gam \ni e^{i  \beta} \mapsto \bg^{\text{odd}}(e^{i  \beta}) :=\langle  g_{-1}(e^{i  \beta}), g_{-3}(e^{i  \beta}), g_{-5}(e^{i  \beta})\cdots \cdot \rangle,
\end{align}
where for each $n \leq -1$ odd, 
\begin{align}\label{gnk_bg1}
g_{n}(e^{i  \beta}):=\sum_{k=-\infty}^\infty g_{n,k} \, e^{i k \beta}.
\end{align}
We construct the sequence valued map $\bu^{\text{odd}} (z)$ inside $\OM$, by the Bukhgeim-Cauchy Integral formula \eqref{BukhgeimCauchyFormula}, namely 
\begin{align}\label{u_construction}
\bu ^{\text{odd}}  (z) = \langle u_{-1}(z), u_{-3}(z), u_{-5}(z),  ... \rangle := \B \left[\bg^{\text{odd}}\right] (z), \quad z\in \OM.
\end{align}
Using the decay assumption \eqref{gnk_decay}, by Lemma \ref{prop:bg_gnk} in the appendix,  $\bg^{\text{odd}} \in l^{1,2}_{\INF}(\Gam) \cap C^{1,\mu}(\Gam; l^1)$, and, in particular,
$\bg^{\text{odd}} \in  Y_{\mu}(\Gam)$ for $\mu>1/2$. 

By Theorem \ref{BukhgeimCauchyThm} (ii), the constructed $\bu ^{\text{odd}} \in C^{1,\mu}(\OM;l^{1}) \cap C^{\mu}(\ol \OM;l^{1})$  is $L^2$-analytic in $\OM$,
\begin{align}\label{vneg}
\ol{\del} u_{n} + \del u_{n-2} = 0,\quad n\leq -1, \; n \; \text{odd}.
\end{align}

While $\bu^{\text{odd}}$ constructed in \eqref{u_construction}  is $L^2$-analytic, in general, its trace $ \bu^{\text{odd}} \lvert_{\Gam}$  need not be equal to $\bg^{\text{odd}}$. 
It is at this point that the constraints   \eqref{RTCond} come  into play.
By using the constraints \eqref{RTCond},
 \begin{align*}
g_{n,k}=  (-1)^k  g_{n+2k,-k}, \quad \text{for } \;   n \leq -1 \,\text{odd}, \;\text{and}\; k\leq -1,
\end{align*} in \eqref{Hu_comp}, 
 we obtain that, for each $n \leq -1$, $n$ odd, and $k \in \BZ$,
$\ds([I+i\HT] \bg^{\text{odd}})_{n,k}= 0.$ Thus, $\ds[I+i\HT] \bg^{\text{odd}}= \bzero$, and the sufficiency part of Theorem \ref{NecSuf_BukhgeimHilbert_Thm} applies to yield
\begin{align}\label{buodd-trace_bgodd}
\bu^{\text{odd}} \lvert_{\Gam} = \bg^{\text{odd}}.
\end{align}

%
%
%

All of the positive Fourier modes $u_n, g_n$ for $n\geq 1$ odd are constructed by conjugation,
\begin{align}\label{construct_vpos}
u_{n}&:=\ol{u_{-n}},\quad \text{in}\; \OM,  \\
g_{n}&:=\ol{g_{-n}}, \quad \text{on}\; \Gam. 
\end{align}
Also, by conjugating \eqref{vneg} we note that the positive Fourier modes satisfy
\begin{align*}
\ol{\del} u_{n+2} + \del u_{n} = 0,\quad n\geq 1, \; n \; \text{odd}.
\end{align*}Moreover, using \eqref{buodd-trace_bgodd} they extend continuously to $\Gam$ and
\begin{align*}
u_{n}|_\Gam=\ol{u_{-n}}|_\Gam=\ol{g_{-n}}=g_{n},\quad n\geq 1, \; n \; \text{odd}.
\end{align*}

In summary, we have shown that 

\begin{alignat}{2}\label{allmodes}
&\ol{\del} u_{n} + \del{u_{n-2}} = 0,  && \quad \text{for all odd integers } n\neq1,\\ \label{alluTrace_g}
&u_{n}\lvert_{\Gam} = g_{n}, && \quad \text{for all odd integers } n.
\end{alignat}

Define next the following three real valued  functions: 
\begin{align}\label{definitionU}
u^{\text{odd}}(z, e^{i\theta})&:=  \sum_{\substack{n=- \INF \\ n = \,\text{odd}}}^{\INF}  u_{n}(z)e^{i n\tta}, \quad (z,\theta)\in \OM \times \sph,
\end{align}
\begin{align}\label{definitiong}
g(z, e^{i\theta})&:=  \sum_{\substack{n=- \INF \\ n = \,\text{odd}}}^{\INF}  g_{n}(z)e^{i n\tta}, \quad (z,\theta)\in\Gam\times\sph,
\end{align} 
and 
\begin{equation}\label{fDefnTh3}
f(z) := \re \left ( \del u_{-1}(z) \right ),\quad z\in\OM.
\end{equation}
Note that $u_{-1}\in C^{1,\mu}(\OM)$, yields $f\in C^\mu(\OM)$.

We are left to prove that $g$ defined in \eqref{definitiong} coincides with $Xf$ on $\Gam_+$. 

Since only odd modes are used in the angular variable, $g(z,\cdot)$ is an odd function, and thus satisfy \eqref{sym**}. We claim that $g$ also satisfies the symmetry condition \eqref{sym*}. Indeed,

\begin{align*}g (e^{i (2 \theta -\beta -\pi)}, e^{i(\theta +\pi)}) &=\sum_{\substack{m=- \INF \\m = \,\text{odd}}}^{\INF}  \sum_{p \in \BZ} g_{m,p} e^{i m (\theta+\pi)}  e^{i p (2 \theta -\beta -\pi)}\\
&=  \sum_{\substack{m=- \INF \\ m = \,\text{odd}}}^{\INF}   \sum_{p \in \BZ} (-1)^{m+p}g_{m,p}  e^{i (2p+m) \theta}  e^{-i p \beta }\\
& \xlongequal{ m=n+2k, \;  p=-k}\sum_{\substack{n=- \INF \\ n = \,\text{odd}}}^{\INF}    \sum_{k \in \BZ}  (-1)^{n+k}g_{n+2k,-k} e^{i n \theta}  e^{i k \beta }\\
& \xlongequal{ \eqref{RT_FourierEven}} \sum_{\substack{n=- \INF \\ n = \,\text{odd}}}^{\INF}   \sum_{k \in \BZ} g_{n,k} e^{i n \theta}  e^{i k \beta }=g(e^{i
\beta},e^{i\theta}).\numberthis\label{g_symm*}
\end{align*}

Since $\bu ^{\text{odd}} \in C^{1,\mu}(\OM;l^{1}) \cap C^{\mu}(\ol \OM;l^{1})$, by \cite[Corollary 4.1]{sadiqTamasan01} and
\cite[Proposition 4.1 (iii)]{sadiqTamasan01}, we conclude that $u^{\text{odd}} \in C^{1,\mu}(\OM \times \sph)\cap C^{\mu}(\ol{\OM}\times \sph)$. In particular,  for each $e^{i\theta}\in \sph$, the trace of $u^{\text{odd}}(\cdot,e^{i\theta})$  satisfies
\begin{align} \label{utrace_g}
u^{\text{odd}}(\cdot,e^{i\theta})\lvert_{\Gam} &=\left.\left ( \sum_{\substack{n= - \INF \\ n = \,\text{odd}}}^{\INF}    u_{n}e^{i n\tta}\right) \right\lvert_{\Gam}
= \sum_{\substack{n= - \INF \\ n = \,\text{odd}}}^{\INF}   \left ( u_{n}\lvert_{\Gam} \right)  e^{i n\tta} 
= \sum_{\substack{n=- \INF \\ n = \,\text{odd}}}^{\INF}    g_{n}e^{i n\tta} = g(\cdot,e^{i\theta}),
\end{align} 
where the third equality above uses \eqref{alluTrace_g}. Since $u^{\text{odd}} \in C^{\mu}(\ol{\OM}\times \sph)$, its trace $g \in C^{\mu}(\Gam \times \sph)$.

Since $u^{\text{odd}}\in C^{1,\mu}(\OM \times \sph)$,  the formal calculation in \eqref{eq:transport_fourier}  is now justified for each $z \in \OM$. 
For each direction $\btheta=\langle\cos\theta,\sin\theta\rangle$, we obtain
\begin{align}\label{transport_uf}
\btheta\cdot\nabla u^{\text{odd}}(z, e^{i\theta})
&=
2 \re \left ( \del u_{-1}(z) \right ) + 2\re\left\{ \sum_{\substack{n=1\\ n = \,\text{odd}}}^{\INF}   \left(\dba u_{-n}(z) + {\del} u_{-n-2}(z)\right)e^{-in\tta}\right\} =2f(z),
\end{align}  
where the second equality uses \eqref{allmodes} and \eqref{fDefnTh3}.

In the following calculation we let $ l(\beta, \theta)  = \lvert e^{i \beta} -e^{i (2\theta -\beta-\pi)}  \rvert$ denote the length of the chord in Figure \ref{fig:fanbeam1}, and use the geometric equality
$e^{i\beta}-l(\beta,\theta)e^{i\theta}=e^{i(2\theta-\beta-\pi)}$. For each $(e^{i \beta},e^{i \theta}) \in \Gam_+$,
\begin{align*}
2 g(e^{i \beta},e^{i \theta}) &= g(e^{i \beta},e^{i \theta})- g(e^{i\beta},e^{i (\theta+\pi)})\\ 
&= g(e^{i \beta},e^{i \theta})- g(e^{i (2\theta -\beta-\pi)},e^{i \theta})\\ 
& =u^{\text{odd}}(e^{i \beta},e^{i \theta})- u^{\text{odd}}(e^{i (2\theta -\beta-\pi)},e^{i \theta}) 
 \\ 
 &
= \int_{-  l(\beta, \theta) }^0  \btheta\cdot\nabla u ^{\text{odd}}(e^{i \beta}+t e^{i \theta}, e^{i\theta})dt\\ 
&=\int_{- l(\beta, \theta) }^0 2f (e^{i \beta}+t e^{i \theta})dt 
= 2[Xf] (e^{i \beta},e^{i \theta}), \numberthis  \label{eq:g_Xf}
\end{align*}
where the first equality uses $g(e^{i\beta},\cdot)$ is angularly odd, the second equality uses the symmetry relation \eqref{g_symm*} with $\theta$ replaced by $\theta+\pi$, 
the third equality uses \eqref{utrace_g},
the fourth equality is the fundamental theorem of calculus, the fifth 
equality uses \eqref{transport_uf}, and the last equality uses the support of $f$ in $\OM$. 

Therefore $g =Xf$ on $\Gam_+$, and, since $g$ is angularly odd, $g=-Xf$ on $\Gam_-$.

The equation \eqref{eq:g_Xf} also shows that $f$ integrates along all lines in the direction of $\btheta$, and thus $f\in L^1(\OM)$.


 \qed  



\section{Connection with the Gelfand-Graev-Helgason-Ludwig moment conditions}



In retrospect, the necessity of the constraints \eqref{RTCond}:
\begin{align*}
g_{n,k}=  (-1)^k  g_{n+2k,-k}, \text{ for all } n\leq -1 \text{ odd,  and all } k\leq 0,
\end{align*}
can be inferred directly from the moment conditions of Gelfand-Graev, Helgason, and Ludwig (GGHL) as follows.

For $\OM$ the unit disc and $f\in L^1(\OM)$,  recall its Radon transform $\displaystyle Rf(s,\bomega)=\int_{-1}^1f(s\bomega+ t\bomega^\perp)dt.$

Original formulation of the GGHL- moment conditions state that: for each integer $p\geq 0$, the map $ \displaystyle\sph\ni\bomega\mapsto \int_{-1}^1s^{p} Rf (s, \bomega)ds$ is a homogeneous polynomial of degree $p$. We work with the equivalent formulation below (e.g., \cite{nattererBook}; see also the appendix).
\begin{prop}[GGHL- moment conditions]\label{HLmoments} 
Let $f\in L^1(\OM)$, and $p\in\BZ$ with $p\geq 0$. Then, 
\begin{equation} \label{eq:Moment_Orthogonality}
\int_{-\pi}^{\pi} e^{i m \omega } \int_{-1}^{1} s^{p} Rf (s, e^{i \omega}) \, ds \, d \omega = 0, \text { for all }
\end{equation}

(i) $m\in\BZ$ with $p-m$ odd, or

(ii) $m\in \BZ$ with $|m|>p$ and $p-m$ even.
\end{prop}


\begin{figure}[ht]
 \centering

  \pgfmathsetmacro{\Radius}{1.35} 
  \pgfmathsetmacro{\RRadius}{1.4*\Radius}
  \pgfmathsetmacro{\Height}{1.25}
  \pgfmathsetmacro{\startAngle}{asin(\Height/\Radius)}
  
   \pgfmathsetmacro{\ptAngle}{\startAngle}
   \pgfmathsetmacro{\alphaAngle}{0.5*\startAngle}
   \pgfmathsetmacro{\thetaAngle}{\startAngle + \alphaAngle} 
   \pgfmathsetmacro{\gammaAngle}{\startAngle + 2*\alphaAngle-180} 
   \pgfmathsetmacro{\omegaAngle}{\thetaAngle - 90} 
   \pgfmathsetmacro{\xbpt}{cos(\ptAngle)*\Radius}
   \pgfmathsetmacro{\ybpt}{sin(\ptAngle)*\Radius} 
   
   \pgfmathsetmacro{\xpt}{cos(\ptAngle)*1.1*\RRadius}
   \pgfmathsetmacro{\ypt}{sin(\ptAngle)*1.1*\RRadius} 
   
   \pgfmathsetmacro{\Lxpt}{\xbpt + cos(\thetaAngle)}
   \pgfmathsetmacro{\Lypt}{\ybpt + sin(\thetaAngle)}
   
   \pgfmathsetmacro{\Lxxpt}{\xbpt - cos(\thetaAngle)*2.1*\Radius}
   \pgfmathsetmacro{\Lyypt}{\ybpt - sin(\thetaAngle)*2.1*\Radius}
   
   \pgfmathsetmacro{\omegaxpt}{cos(\omegaAngle)*\RRadius}
   \pgfmathsetmacro{\omegaypt}{sin(\omegaAngle)*\RRadius}

   \pgfmathsetmacro{\gammaxpt}{cos(\gammaAngle)*\Radius}
   \pgfmathsetmacro{\gammaypt}{sin(\gammaAngle)*\Radius}
   
   \pgfmathsetmacro{\sxpt}{cos(\omegaAngle)*0.37*\RRadius}
   \pgfmathsetmacro{\sypt}{sin(\omegaAngle)*0.37*\RRadius}

\begin{tikzpicture}[scale=1.5,cap=round,>=latex]

\tikzset{
    thick/.style=      {line width=0.8pt},
    very thick/.style= {line width=1.1pt},
    ultra thick/.style={line width=1.6pt}
}

  \coordinate[label=below:$0$] (O) at (0,0,0);
  \filldraw[black] (O) circle(1.2pt);
  \draw[thick] (O) circle (\Radius cm);
  
  \coordinate[label=above:$\Omega$] (OM) at(120:0.85cm);
  
  \coordinate (E) at (\RRadius,0,0);
  \draw[dashed] (O) -- (E);

  \coordinate (P) at (\xbpt,\ybpt);
  \filldraw[black] (P) circle(1.2pt);
  
  \coordinate[label=right:$e^{ i \beta}$] (Z) at (\xbpt+0.025,\ybpt+0.1);


  \coordinate[label=above:$\nu$] (nu) at (\xpt,\ypt);
  \draw[thick,->] (O) -- (nu);
  
  \coordinate[label=above:$\btheta$] (L) at  (\Lxpt,\Lypt);
  \draw[thick,->] (P) -- (L);
  
  \coordinate (LL) at  (\Lxxpt,\Lyypt);
  \draw[thick,-] (P) -- (LL);
  
  \coordinate[label=below:$\bomega$] (W) at  (\omegaxpt,\omegaypt);
  \coordinate[label=above:$\nu$] (nu) at (\xpt,\ypt);
  \draw[thick,->] (O) -- (W);
  
    \tikzset{
	position label/.style={
	  above = 3pt,
	  text height = 1.5ex,
	  text depth = 1ex
	},
      brace/.style={
	decoration={brace,mirror},
	decorate
      }
    }  
    \coordinate (S) at  (\sxpt,\sypt);
  \draw [brace,decoration={raise=0.4ex}] [brace]  (O.north) -- (S.north) node [position label, below=3.8, pos=0.4, rotate = 15,scale=0.9] {$s$};
  \pic[draw,->,angle radius=.7cm,angle eccentricity=1.3,"$\beta$"] {angle=E--O--P}; 
  \pic[draw,->,angle radius=.7cm,angle eccentricity=1.3,"$\alpha$"] {angle=nu--P--L}; 
  \pic[draw,-,angle radius=.7cm,angle eccentricity=1.3,"$\alpha$"] {angle=O--P--LL};

\end{tikzpicture}
 \caption{$ e^{i \beta}\in \Gam, \;  \btheta=(\cos \theta, \sin \theta), \;   e^{i \theta} = e^{i(\alpha +\beta)} = \bomega^{\perp}, \; \bomega= e^{i(\alpha +\beta - \frac{\pi}{2})}.$
 }
\label{fig:fanbeam}
\end{figure}
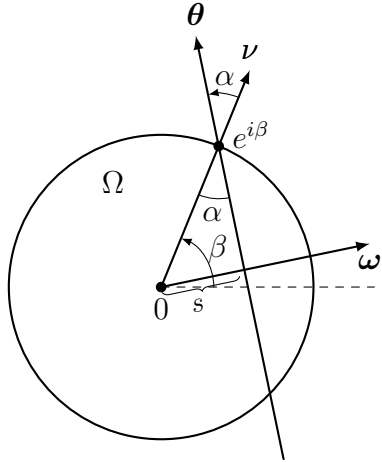

In two dimensions, the $X$-ray and the Radon transform of $f$ are connected by the reparametrization
\begin{align}\label{eq:XrayRadon}
Xf (e^{i \beta}, e^{i(\alpha +\beta)})= Rf(\sin \alpha, e^{i(\alpha +\beta - \frac{\pi}{2})})  \quad \text{ for } \;  |\alpha| \leq \frac{\pi}{2}\text{ and } \beta\in(-\pi,\pi].
\end{align}
Note that \eqref{eq:XrayRadon} only accounts for points on the torus in $\Gam_+$. However, this will be sufficient since
 $Xf$ is an angularly even function satisfying  \eqref{X_ray_sym**}. 
 
 Recall that the conditions \eqref{RTCond} refer only to the odd angular modes. Since the odd angular modes are preserved upon addition with the modes of an angularly even function, suffices to prove  \eqref{RTCond} for
\begin{equation}\label{g0}
g := Xf + \left\{
		\begin{array}{ll}
			Xf \text{ on }\Gam_+,\\
			-Xf \text{ on }\Gam_-
		\end{array}
	\right.
	=
	\left\{
		\begin{array}{ll}
			2\left[Xf \right]\text{ on }\Gam_+,\\
			0 \text{ on }\Gam_-.
		\end{array}
	\right.
\end{equation}

Using the relation \eqref {eq:XrayRadon}, the change of variable $s=\sin\alpha$ in \eqref{eq:Moment_Orthogonality}, and the fact that $g$ vanishes on $\Gam_-$, one easily obtains the reformulation of the moment conditions below.

\begin{cor}[X-ray moment conditions]\label{Xmoments} 
For $f\in L^1(\OM)$, let $g\in L^1(\Gam\times\sph)$ be defined by \eqref{g0}, and $p\in\BZ$ with $p\geq 0$. Then, 
\begin{equation} \label{eq:MomentCond_fanbeam}
\int_{-\pi}^{\pi} \int_{-\pi}^{\pi} e^{i m (\beta +\alpha)} (\sin \alpha)^{p} \, \cos \alpha \, g (e^{i \beta}, e^{i(\alpha +\beta)})   \,d \alpha \, d \beta=0, \text{ for all }
\end{equation}

(i) $m\in\BZ$ with $p-m$ odd, or

(ii) $m\in \BZ$ with $|m|>p$ and $p-m$ even.
\end{cor}

\begin{remark}\label{GGHL=RTCond}
The GGHL-moment conditions as reformulated in \eqref{eq:MomentCond_fanbeam} yield the moment conditions in \eqref{RTCond}. 
\end{remark}
\begin{proof}

We use Corollary \ref{Xmoments} part (ii) with $|m|>p$, and two separate cases:  $p$ and $m$ both even, and $p$ and $m$ both odd. 

Case 1:   $ |m| > p$, and $p$ and $m$ both even. 

Since for all $n\geq 0$, 
$\displaystyle span \left\{ \cos \alpha \, (\sin \alpha)^{2j},  \; 0 \leq j \leq n \right\} = span \left\{ \cos [(2j+1)\alpha], \; 0 \leq j \leq n \right\}$,
the orthogonality in \eqref{eq:MomentCond_fanbeam} for this case becomes 
\begin{equation} \label{eq:MomentCond_pEvenMeven}
\begin{aligned}
0&=\int_{-\pi}^{\pi} \int_{-\pi}^{\pi} e^{i m (\beta +\alpha)  }  \cos [(p+1)\alpha] \; g(e^{i \beta}, e^{i(\alpha +\beta)})  \,d \alpha \, d \beta\\
&=\frac{1}{2} \int_{-\pi}^{\pi} \int_{-\pi}^{\pi} e^{i m (\beta +\alpha)  } e^{i(p+1)\alpha} g(e^{i \beta}, e^{i(\alpha +\beta)})  \,d \alpha \, d \beta 
\\
&\qquad
+ \frac{1}{2}\int_{-\pi}^{\pi} \int_{-\pi}^{\pi} e^{i m (\beta +\alpha)  } e^{-i(p+1)\alpha} g(e^{i \beta}, e^{i(\alpha +\beta)})  \,d \alpha \, d \beta.
\end{aligned}
\end{equation}

In the last equality of \eqref{eq:MomentCond_pEvenMeven} let us consider the first term,
\begin{equation} \label{eq:MC_pEvenMeven2_first}
\begin{aligned}
&\int_{-\pi}^{\pi} \int_{-\pi}^{\pi} e^{i m (\beta +\alpha)  } e^{i(p+1)\alpha} g(e^{i \beta}, e^{i(\alpha +\beta)})  \,d \alpha \, d \beta  \\
&\quad \xlongequal{ \theta = \alpha +\beta} \int_{-\pi}^{\pi} \int_{-\pi +\beta}^{\pi+ \beta}  e^{i m \theta  } e^{i(p+1)(\theta - \beta )}   g(e^{i \beta}, e^{i \theta })  \,d \theta \, d \beta  \\
&\quad= \int_{-\pi}^{\pi} \int_{-\pi}^{\pi}  e^{i \theta (m+p+1)  } e^{- i(p+1) \beta }    g(e^{i \beta}, e^{i \theta })   \,d \theta \, d \beta  = (2 \pi)^2 g_{-m-p-1,p+1}.
\end{aligned}
\end{equation}
Similarly, the last term in \eqref{eq:MomentCond_pEvenMeven} rewrites
\begin{equation} \label{eq:MC_pEvenMeven2_sec}
\begin{aligned}
&\int_{-\pi}^{\pi} \int_{-\pi}^{\pi} e^{i m (\beta +\alpha)  } e^{-i(p+1)\alpha} g(e^{i \beta}, e^{i(\alpha +\beta)})  \,d \alpha \, d \beta \\
&\quad \xlongequal{ \theta = \alpha +\beta} \int_{-\pi}^{\pi} \int_{-\pi +\beta}^{\pi+ \beta}  e^{i m \theta  } e^{-i(p+1)(\theta - \beta )}    g(e^{i \beta}, e^{i \theta })  \,d \theta \, d \beta  \\
&\quad= \int_{-\pi}^{\pi} \int_{-\pi}^{\pi}  e^{i \theta (m-p-1)  } e^{ i(p+1) \beta }    g(e^{i \beta}, e^{i \theta })  \,d \theta \, d \beta  = (2 \pi)^2 g_{-m+p+1,-p-1}.
\end{aligned}
\end{equation}
Using \eqref{eq:MC_pEvenMeven2_first}, and \eqref{eq:MC_pEvenMeven2_sec}, the expression in \eqref{eq:MomentCond_pEvenMeven} yields
\begin{equation*}
\begin{aligned}
g_{-m-p-1,p+1} = -g_{-m+p+1,-p-1}, \quad \text{for} \; m,p \; \text{even,} \; |m|>p, \;\text{and } p \geq 0.
\end{aligned}
\end{equation*}
By setting $k=-p-1$ (odd) and $|m| \geq p+2 = -k+1$, and  $n=-m-k$ (odd), we obtain, in particular,
\begin{equation} \label{eq:MC_nOdd_kOdd_condition}
\begin{aligned}
g_{n,k} = (-1)^k g_{n+2k,-k}, \text{ for all odd } k \leq -1,
	\text{ and all odd }  n\leq -1.	
\end{aligned}
\end{equation}





Case 2: We consider \eqref{eq:MomentCond_fanbeam}
 for all $ p  \geq 0, \;  |m| > p,$ and {$p$ and $m$ both odd}.
 
Since for all $n\geq 0$, 
$\displaystyle span \left\{ \cos \alpha \, (\sin \alpha)^{2j+1}, \; 0 \leq j \leq n \right\} = span \left\{ \sin [(2j+2)\alpha], \; 0 \leq j \leq n \right\}$,
the orthogonality in \eqref{eq:MomentCond_fanbeam} for this case becomes 
 \begin{equation} \label{eq:MomentCond_pOddModd}
\begin{aligned}
0&=\int_{-\pi}^{\pi} \int_{-\pi}^{\pi} e^{i m (\beta +\alpha)  }  \sin [(p+1)\alpha] \; g(e^{i \beta}, e^{i(\alpha +\beta)})  \,d \alpha \, d \beta\\
&=\frac{1}{2 i} \int_{-\pi}^{\pi} \int_{-\pi}^{\pi} e^{i m (\beta +\alpha)  } e^{i(p+1)\alpha} g(e^{i \beta}, e^{i(\alpha +\beta)})  \,d \alpha \, d \beta 
\\
&\qquad
- \frac{1}{2i}\int_{-\pi}^{\pi} \int_{-\pi}^{\pi} e^{i m (\beta +\alpha)  } e^{-i(p+1)\alpha} g(e^{i \beta}, e^{i(\alpha +\beta)})  \,d \alpha \, d \beta.
\end{aligned}
\end{equation}


Using \eqref{eq:MC_pEvenMeven2_first} and \eqref{eq:MC_pEvenMeven2_sec}, 
the expression in \eqref{eq:MomentCond_pOddModd}  yields
\begin{equation*}
\begin{aligned}
g_{-m-p-1,p+1} = g_{-m+p+1,-p-1}, \quad \text{for} \; m,p \; \text{odd,} \; |m|>p, \;\text{and } p \geq 0.
\end{aligned}
\end{equation*}By setting $k=-p-1$ (even), $|m| \geq p+2 = -k+1$, and  $n=-m-k$ (odd), we obtain 
\begin{equation} \label{eq:MC_nOdd_kEven_condition}
\begin{aligned}
g_{n,k} = (-1)^k g_{n+2k,-k}, \text{ for all even } k \leq -2, \text{ and all odd } n\leq -1.	
\end{aligned}
\end{equation}

Using \eqref{eq:MC_nOdd_kOdd_condition} and \eqref{eq:MC_nOdd_kEven_condition}, 
\begin{align*}
g_{n,k} = (-1)^k g_{n+2k,-k}, \text{ for all } k \leq -1, \text{ and all odd }n\leq -1.
\end{align*}

The relation above trivially holds for $k=0$.  
\end{proof}

In proving \eqref{RTCond}, we used only the part (ii) of Proposition \ref{HLmoments}. Part (i) of Proposition \ref{HLmoments} yields relations already included in the symmetry \eqref{RT_FourierEven}.

\section*{Acknowledgment}
The work of K.~ Sadiq  was supported by the Austrian Science Fund (FWF), Project P31053-N32, and by the FWF Project F6801–N36 within the Special Research Program SFB F68 “Tomography Across the Scales”. 
The work of A.~Tamasan  was supported in part by the National Science Foundation DMS-1907097.

\appendix

\section{Elementary results}
To improve the readability, we moved the proof of the more elementary claims to this section. The presentation follows the order of their occurrence. Recall  $\OM = \{ z  \in \BC : |z| <1  \}$ is the complex unit disc,  $\Gam = \{ z  \in \BC : |z| =1  \}$ is its boundary, and $\sph$ is the set of unit directions.

\begin{prop}\label{prop:fLp_regularity}
	Let $f \in L^1(\OM)$. If
	\begin{align*}
		\textnormal{supp }f  \subset \{ z: \lvert z \rvert \leq\sqrt{ 1- \delta^2}  \}, \;  0 < \delta <1, \quad \textnormal{or} \quad  f \in L^p(\OM), \;p >2,
	\end{align*}
	then $Xf  \in L^1(\Gamma\times\sph)$.
\end{prop}

\begin{proof}
	For $e^{i \beta} \in \Gam$ and $e^{i\theta} \in \sph$,  the $X$-ray transform of $f$ (extended by 0 outside $\OM$) is given by
	\begin{align*}
		Xf(e^{i \beta}, e^{i \theta}) = \int_{-\INF}^{\INF} f(e^{i \beta} +t e^{i \theta}) dt= \int_{-\INF}^{\INF} f_{\theta}(e^{i (\beta-\theta)} +t ) dt,
	\end{align*}	where $$f_\theta (z) := f(ze^{i \theta})$$is obtained from $f$ by a rotation by angle $\theta$. 
	
	We estimate 
	\begin{align*}
		\lnorm{Xf}&_{L^1(\Gamma\times\sph)} = \frac{1}{(2\pi)^2}    \int_{-\pi}^{\pi}
		\int_{-\pi}^{\pi} \lvert  Xf(e^{i \beta}, e^{i \theta}) \rvert d \beta d \theta \\
		& \leq  \frac{1}{(2\pi)^2}  \int_{-\pi}^{\pi}
		\int_{-\pi}^{\pi} \int_{-\INF}^{\INF} \lvert  f_{\theta}(e^{i (\beta-\theta)} +t ) \rvert  dt  d \beta d \theta  \\
		&  \xlongequal{ \alpha = \beta-\theta } 2\frac{1}{(2\pi)^2}    \int_{-\pi}^{\pi}
		\int_{-\pi/2}^{\pi/2} \int_{-\INF}^{\INF} \lvert  f_{\theta}(e^{i \alpha} +t ) \rvert  dt d \alpha d \theta \\
		&  \xlongequal{s = \sin \alpha } \frac{1}{2\pi^2}    \int_{-\pi}^{\pi}
		\int_{-\INF}^{\INF} \int_{-1}^{1}  \frac{\lvert f_{\theta}(\sqrt{1-s^2}+t + is ) \rvert }{\sqrt{1-s^2}}  ds dt d \theta \\
		&  \xlongequal{u = t + \sqrt{1-s^2} }\frac{1}{2\pi^2}    \int_{-\pi}^{\pi}
		\int_{-\INF}^{\INF} \int_{-1}^{1} \frac{ \lvert  f_{\theta}(u +is ) \rvert  }{\sqrt{1-s^2}}    ds du d \theta \\  \numberthis \label{eq:Xfnorm}
		&	= \frac{1}{2\pi^2}    \int_{-\pi}^{\pi}
		\int_{-1}^{1} \int_{-1}^{1} \frac{ \lvert  f_{\theta}(u+i s ) \rvert  }{\sqrt{1-s^2}}   ds du d \theta,
	\end{align*}	where the last equality uses $ \supp f_\theta \subset \OM$.
	
	If $\supp f  \subset  \{ z: \lvert z \rvert \leq\sqrt{ 1- \delta^2}  \}$,  then 
	\begin{align*}
		\lnorm{Xf}_{L^1(\Gamma\times\sph)} & \leq \frac{1}{2\pi^2}    \int_{-\pi}^{\pi}
		\int_{-1}^{1} \int_{-\sqrt{1-\delta^2}}^{\sqrt{1-\delta^2}} \frac{ \lvert  f_{\theta}(u+is ) \rvert  }{\sqrt{1-s^2}}    ds du d \theta \\
		& \leq \frac{1}{2\pi^2}    \frac{1}{\delta} \int_{-\pi}^{\pi}
		\lnorm{f_{\theta}}_{L^1(\OM)} d \theta 
		= \frac{1}{\pi \delta }  \lnorm{f}_{L^1(\OM)}.
	\end{align*}	
	
	Next we  consider  $f \in L^p(\OM)$, $p >2$. Let also $T:=(-1,1)\times(-1,1)$ denote the unit square. 
	Since $\OM\subset T$, for every $\theta\in(-\pi,\pi]$, $\ds f_\theta\in L^p(T)$, and  $\ds \lVert f_\theta\lVert_{L^p(T)}=\lVert f\lVert_{L^p(\OM)}$. 
	
	Let $q=\frac{p}{p-1}$ be the conjugate index of $p$. Since $p>2$, $q<2$, the map $\ds  T  \ni(u,s) \mapsto \frac{ 1 }{\sqrt{1-s^2}}$ is in $ L^q(T)$, and 
	$\ds \lnorm{\frac{ 1}{\sqrt{1-(\cdot)^2}}}_{L^q(T)}=2\lnorm{\frac{ 1}{\sqrt{1-(\cdot)^2}}}_{L^q( -1,1)}$. An application of the Hölder's inequality in \eqref{eq:Xfnorm} yields
	\begin{align*}
		\lnorm{Xf}_{L^1(\Gamma\times\sph)} & \leq \frac{1}{\pi^2} \int_{-\pi}^{\pi}
		\lnorm{\frac{ 1}{\sqrt{1-(\cdot)^2}}}_{L^q(T)} \lnorm{f_{\theta}}_{L^p(\OM)} d \theta 
		= \frac{2}{\pi} \lnorm{\frac{ 1}{\sqrt{1-(\cdot)^2}}}_{L^q( -1,1)}
		\lnorm{f}_{L^p(\OM)} .
	\end{align*}

\end{proof}


\begin{lemma}\label{lem:evenness}
	Let $g$ be an integrable function on the torus satisfying the symmetry relation
	\begin{align}\label{sym*1}
	g (e^{i \beta},e^{i \theta})=g (e^{i(2\theta -\beta-\pi)},e^{i( \theta+\pi)}),\text{ for } (e^{i \beta},e^{i \theta})\in\Gam\times\sph,
	\end{align}
		and $g_{n,k}$'s 
		be its Fourier coefficients. Then 
	\begin{align}\label{FourierEvenness_cond}
	g_{n,k}=(-1)^{n+k}g_{n+2k,-k}, \mbox{ for all }n,k\in\BZ.
	\end{align} 
\end{lemma}

\begin{proof}
	
	Indeed,
	\begin{align*}
	g_{n,k} &= \frac{1}{(2 \pi)^2} \int_{-\pi}^{\pi} \int_{-\pi}^{\pi} g(e^{i \beta}, e^{i\theta}) e^{-i n \theta}  e^{-i k \beta} d \theta d \beta \\
	&\xlongequal{  \eqref{sym*1} } 
	\frac{1}{(2 \pi)^2} \int_{-\pi}^{\pi} \int_{-\pi}^{\pi} g (e^{i(2\theta -\beta-\pi)},e^{i( \theta+\pi)}) e^{-i n \theta}  e^{-i k \beta} d \theta d \beta \\
	& \xlongequal{ \gamma = \theta +\pi } (-1)^{ n}
	\frac{1}{(2 \pi)^2} \int_{-\pi}^{\pi} \int_{0}^{2\pi} g (e^{i(2\gamma -\beta-3\pi)},e^{i \gamma})  e^{-i n \gamma}  e^{-i k \beta} d \gamma d \beta \\
	&\xlongequal{ \text{periodicity} }  (-1)^{ n}
	\frac{1}{(2 \pi)^2} \int_{-\pi}^{\pi} \int_{-\pi}^{\pi} g (e^{i(2\gamma -\beta-\pi)},e^{i \gamma})  e^{-i n \gamma}  e^{-i k \beta} d \gamma d \beta \\
	&\xlongequal{ \alpha = 2\gamma -\beta-\pi}  (-1)^{ n}
	\frac{1}{(2 \pi)^2} \int_{-\pi}^{\pi} \int_{2\gamma}^{2\gamma-2\pi} g (e^{i \alpha},e^{i \gamma})  e^{-i k (2\gamma-\alpha-\pi)} e^{-i n \gamma} (-d \alpha) d \gamma \\
	&=  (-1)^{ n+k}	 \frac{1}{(2 \pi)^2} \int_{-\pi}^{\pi} \int_{2\gamma-2\pi}^{2\gamma} g (e^{i \alpha},e^{i \gamma})   e^{-i (n+2k)  \gamma} e^{i k \alpha} d \alpha d \gamma \\
	&\xlongequal{ \text{periodicity} } 
	(-1)^{ n+k}
	\frac{1}{(2 \pi)^2} \int_{-\pi}^{\pi}   \int_{-\pi}^{\pi} g (e^{i \alpha},e^{i \gamma})    e^{-i (n+2k)  \gamma} e^{i k \alpha} d \alpha  d \gamma \\
	&=(-1)^{n+k}g_{n+2k,-k}.
	\end{align*} 
\end{proof}

The following result is used in the proof of Theorem \ref{newmapping_Hilbert}.
\begin{lemma}\label{singularintegral}
	For $\beta \in (-\pi,\pi],$
	\begin{align}\label{Fmodes}
		\frac{1}{\pi} \int_{-\pi}^{\pi} \frac{e^{in\alpha}}{e^{i\alpha}-e^{i\beta}} d \alpha=
		\left\{
		\begin{array}{ll}e^{i(n-1)\beta}&\mbox{ if } n\geq 1,\\
			-e^{i(n-1)\beta}&\mbox{ if } n\leq 0.
		\end{array}
		\right.
	\end{align}
\end{lemma}
\begin{proof}
	Let  $\OM = \{ z  \in \BC : |z| <1  \}$ be the complex unit disc, and $\Gam = \{ z  \in \BC : |z| =1  \}$ be  its boundary.
	\begin{enumerate}
		\item[I.] Let $n \geq 0$.
		For $z \in \OM$, apply the Cauchy integral formula to $z \mapsto z^{n}$, to get 
		\begin{align}\label{cauchy_npos}
			z^{n} = \frac{1}{2\pi i} \int_{\Gam} \frac{\zeta^{n}}{\zeta - z} d \zeta.
		\end{align}
		Let $z \mapsto z_0\in \Gam$, with $|z|<1$, and apply Plemelj-Sokhotski formula (e.g \cite[Chap. 2, Section 17]{muskhellishvili}) on both sides of \eqref{cauchy_npos},
		\begin{align*}
			z_0^{n} &= \frac{1}{2} z_0^{n}+ \frac{1}{2\pi i} \int_{\Gam} \frac{\zeta^{n}}{\zeta - z_0} d \zeta, \\
			z_0^{n} &=  \frac{1}{\pi i} \int_{\Gam} \frac{\zeta^{n}}{\zeta - z_0} d \zeta,
		\end{align*} where the integral is understood in the Cauchy principal value sense.
		Now parametrize the unit circle $\del \OM$.
		In parametric form, let $z_0= e^{i \beta}$, $\zeta = e^{i \alpha}$, then $d \zeta = i e^{i \alpha} d \alpha$, and 
		\begin{align*}
			e^{i n \beta}&=  \frac{1}{\pi i} \int_{-\pi}^{\pi} \frac{e^{i n \alpha} i e^{i \alpha}}{e^{i \alpha} - e^{i \beta}} d \alpha = \frac{1}{\pi } \int_{-\pi}^{\pi} \frac{e^{i (n+1) \alpha} }{e^{i \alpha} - e^{i \beta}} d \alpha, \quad n \geq 0,
		\end{align*}
		Thus, 
		\begin{align*}
			e^{i (n-1) \beta}&=  \frac{1}{\pi } \int_{-\pi}^{\pi} \frac{e^{i n \alpha} }{e^{i \alpha} - e^{i \beta}} d \alpha, \quad n \geq 1.
		\end{align*}
		
		\item[II.] 
		For  $n\geq 0$, we apply the Cauchy integral formula to  $z\mapsto z^{-(n+1)}$ analytic in $(\BC \cup \INF) -\OM$,
		\begin{align}\label{cauchy_ext}
			z^{-(n+1)} = -\frac{1}{2\pi i} \int_{\Gam} \frac{\zeta^{-(n+1)}}{\zeta-z } d \zeta,
		\end{align} where the integral is counter-clockwise, hence the (-) sign to account for the exterior of the disc.
		Let $z \mapsto z_0\in \Gam$, with $|z|>1$, and apply the exterior Plemelj-Sokhotski formula (e.g \cite[Chap. 2, Section 17]{muskhellishvili}) on both sides of \eqref{cauchy_ext}, to get
		\begin{align} \nonumber
			z_0^{-(n+1)} &= -\left[-\frac{1}{2} z_0^{-(n+1)}+  \frac{1}{2\pi i} \int_{\Gam} \frac{z^{-(n+1)}}{z- z_0 } d z \right], \\ \label{blah}
			-z_0^{-(n+1)} &= \frac{1}{\pi i} \int_{\Gam} \frac{\zeta^{-(n+1)}}{\zeta- z_0 } d \zeta, 
		\end{align} where the integral is understood in the Cauchy principal value sense.
		Now parametrize the unit circle by $\zeta = e^{i \alpha}$, and let $z_0 = e^{i \beta}$. Then \eqref{blah} becomes 
		\begin{align*}
			-e^{-i (n+1) \beta}&=  \frac{1}{\pi i} \int_{-\pi}^{\pi} \frac{e^{-i (n+1) \alpha} i e^{i \alpha}}{e^{i \alpha} - e^{i \beta}} d \alpha = \frac{1}{\pi } \int_{-\pi}^{\pi} \frac{e^{-i n \alpha} }{e^{i \alpha} - e^{i \beta}} d \alpha, \quad n \geq 0.
		\end{align*}
		The change $n \to -n$, in the above integral show \eqref{Fmodes}  for $n \leq 0$.
	\end{enumerate}
\end{proof}

The following result is a direct consequence of  the uniform convergence. Recall that 
$\sph$ denotes the unit circle.
\begin{lemma}\label{lem_nfn} 
	Let $(X,\lnorm{\cdot})$ be a Banach space, $\{a_n\}_{n \in \BZ}$ be a sequence in $X$, and $p\geq 0$ integer.
If $\ds \sum_{n \in \BZ} \jpn^p \lnorm{ a_n} < \INF$, then $ \ds \sph \ni \zeta \mapsto  \sum_{n \in \BZ} a_n \zeta^{ n}$ defines a $C^p$ map on $\sph$  with values in $X$.
\end{lemma}

For the following result, we recall some of the spaces in \eqref{spaces}, for $0<\mu<1$, $p=1,2$:
\begin{equation}\label{spaces1}
\begin{aligned} 
l^{1,p}_{\INF}(\Gam) &:= \left \{ \bg:=\langle  
g_{0}, g_{-1} , g_{-2}, \cdots  \rangle\; : \lnorm{\bg}_{l^{1,p}_{\INF}(\Gam)}:= \sup_{\xi \in \Gam}\sum_{j=0}^{\INF}  \jpj^p \lvert g_{-j}(\xi) \rvert < \INF \right \},\\
C^{\mu}(\Gam; l_1) &:= \left \{ \bg:=\langle  
g_{0}, g_{-1} , g_{-2}, \cdots  \rangle:
\sup_{\xi\in \Gam} \lVert \bg(\xi)\rVert_{\ds l_{1}} + \underset{{\substack{
			\xi,\eta \in \Gam \\
			\xi\neq \eta } }}{\sup}
\frac{\lVert \bg(\xi) - \bg(\eta)\rVert_{\ds l_{1}}}{|\xi - \eta|^{ \mu}} < \INF \right \}, 
\end{aligned}
\end{equation} where, for brevity, we use the notation $\jpj=(1+|j|^2)^{1/2}$.

\begin{lemma}\label{prop:bg_gnk}  
	Let $\{g_{-n,k}\}_{n \geq 0,  k\in \BZ}$ be a double sequence  satisfying  the decay
	\begin{align}\label{gnk_decay_1}
	\sum_{n=0}^{\INF}  \jpn^{2} \sum_{k=-\INF}^{\INF} \lvert g_{-n,k} \rvert < \INF, \quad \text{and} \quad
	\sum_{k=-\INF}^{\INF}  \jpk^{1+\mu} 
	\sum_{n=0}^{\INF} \lvert g_{-n,k} \rvert < \INF,  \; \text{for some} \; 0\leq \mu \leq1.
	\end{align}
    Let  $\ds  \bg( \zeta) :=\langle  
	g_{0}( \zeta ), g_{-1}(\zeta ), g_{-2}(\zeta ), \cdots \cdot \rangle,$ for $\zeta \in \Gam=\{z: |z|=1\}$, and where for each $n \geq 0$,  
	$\ds g_{-n}(\zeta):=\sum_{k=-\infty}^\infty g_{-n,k} \, \zeta^{ k }.$
	Then $\bg \in l^{1,2}_{\INF}(\Gam) \cap C^{1,\mu}(\Gam; l^1)$.
\end{lemma}
\begin{proof}
	The fact that $\bg \in l^{1,2}_{\INF}(\Gam)$ follows from the first bound in \eqref{gnk_decay_1}:
	\begin{align*}
	\lnorm{\bg}_{l^{1,2}_{\INF}(\Gam)} = \sup_{\zeta  \in \Gam}
	\sum_{n=0}^{\INF}   \jpn^2 \lvert g_{-n}( \zeta ) \rvert 
	\leq  \sum_{n=0}^{\INF} \jpn^2  \sum_{k=-\INF}^{\INF} \left \lvert g_{-n,k} \right \rvert
	< \INF.
	\end{align*}
	We prove that $\bg \in C^{1,\mu}(\Gam; l^1)$ by interpolation between  $C^{1}(\Gam; l^1)$, and  $C^{1,1}(\Gam; l^1)$. 
	For each $k \in \BZ$, consider the $l^1$ sequence  $\bF(k) := \{g_{-n,k}\}_{n \geq0}$.
	By the second bound in \eqref{gnk_decay_1}, we have 
	\begin{align*}
	\sum_{k\in \BZ}  \jpk^{1+\mu} \lnorm{\bF(k)}_{l^1} = \sum_{k=-\INF}^{\INF}   \jpk^{1+\mu} \sum_{n=0}^{\INF} \left \lvert g_{-n,k} \right \rvert
	< \INF.
	\end{align*}

Since $\{\jpk^{1+\mu} \bF(k) \}_{k \in \BZ}  \in l^1(\BZ_k;l^1(\BN_n))$, by applying Lemma \ref{lem_nfn} for $X=l^1(\BN_n)$, and $p = 1+\mu $, with $\mu =0$, we conclude  $\bg \in C^{1}(\Gam; l^1)$, whereas with $\mu =1$, we obtain  $\bg \in C^{1,1}(\Gam; l^1)$.
The result for $0 < \mu <1$ follows by interpolation.

\end{proof}

\begin{prop}\label{HLmomentsApp}[GGHL- moment conditions]
Let $f\in L^1(\OM)$, and $p\in\BZ$ with $p\geq 0$. Then, 
\begin{equation} \label{eq:Moment_OrthogonalityApp}
\int_{-\pi}^{\pi} e^{i m \omega } \int_{-1}^{1} s^{p} Rf (s, e^{i \omega}) \, ds \, d \omega = 0, \text { for all }
\end{equation}

(i) $m\in\BZ$ with $p-m$ odd, or

(ii) $m\in \BZ$ with $|m|>p$ and $p-m$ even.
\end{prop}

\begin{proof}

\begin{equation*} 
\begin{aligned}
\int_{-1}^1s^{p} Rf (s, e^{i\omega})ds &= \int_{-1}^1 s^{p} \int_{-1}^{1} f(s e^{i\omega} +t  e^{i(\omega+\pi/2)}) dt ds 
\\
& = \int_{-1}^1\int_{-1}^1 (x_1 \cos\omega +x_2 \sin\omega)^p f(x_1,x_2) dx_1 dx_2 \\
&= \sum_{k=0}^{p} \binom{p}{j} \left( \int_{-1}^1\int_{-1}^1 f(x_1,x_2) x_1^{p-j} x_2^j \, dx_1 \, dx_2 \right )(\cos \omega)^{p-k} (\sin\omega)^k \\
&= \sum_{k=0}^{p} f_{p,k} \, \left(\frac{e^{i\omega}+e^{-i\omega}}{2}\right)^{p-k}  \left(\frac{e^{i\omega}-e^{-i\omega}}{2i}\right)^k \\
&= 2^{-p}e^{-i \omega p }\sum_{j=0}^p (-i)^k f_{p,k}  \left( (e^{i \omega})^2+ 1 \right)^{p-k} \left( (e^{i \omega})^2 - 1 \right)^{k} \\
&=  2^{-p} e^{-ip\omega}\sum_{j=0}^{p}  (-i)^k f_{p,k}  Q_{p,k} (e^{2i\omega}),
\end{aligned}
\end{equation*}
where $Q_{p.k} (t) =  \left( t+ 1 \right)^{p-k} \left( t - 1 \right)^k$, and $\ds f_{p,k} =  \binom{p}{k}  \int_{-1}^1\int_{-1}^1 f(x_1,x_2) x_1^{p-k} x_2^k \, dx_1 \, dx_2.\\$

Thus,  the map $\ds\zeta\mapsto \zeta^{p}\int_{\BR}s^{p} Rf (s, \zeta)ds$ is a polynomial of degree $2p$ in $\zeta\in\BC$ with even powers only.  
In particular, we get the orthogonality (in $L^2(\sph)$) conditions:
\begin{equation*}
\begin{aligned}
\int_{-\pi}^{\pi} e^{i \omega (p-n)} \int_{-1}^{1} s^{p} Rf (s, e^{i \omega}) \, ds \, d \omega = 0, \; 
\left\{
                \begin{array}{ll}
                  \text{for} \; n <0, \text{ or } n>2p,\\
                  \text{or for} \; 1 \leq n \leq 2p-1, \; \text{and} \; n \;\text{odd}.
                \end{array}
              \right.
\end{aligned}
\end{equation*}
By setting $m = p-n$, 

\begin{equation*} 
\begin{aligned}
\int_{-\pi}^{\pi} e^{i m \omega } \int_{-1}^{1} s^{p} Rf (s, e^{i \omega}) \, ds \, d \omega = 0, \; 
\left\{
                \begin{array}{ll}
                  \text{for} \; |m|  >p,\\
                  \text{or for} \;  |m|  \leq p, \; \text{and} \; p-m \;\text{is odd},
                \end{array}
              \right.
\end{aligned}
\end{equation*}
The last conditions are the same as in \eqref{eq:Moment_OrthogonalityApp}.

\end{proof}


\end{document}